\documentclass[12pt, reqno]{amsart}
\usepackage{amssymb,amsmath,amsthm,newlfont}
\usepackage{psfrag}
\usepackage[dvips,final]{graphicx}
\usepackage[T1]{fontenc}
\usepackage{color}
\usepackage{a4wide}
\usepackage[utf8]{inputenc}

\newcommand{\dist}{\mathop{\textmd{dist}}}
\newcommand{\Ll}{L} %\textmd{L}\{L}\{L}\{L}

\numberwithin{equation}{section}

\newcommand{\R}{\mathbb{R}}
\newcommand{\C}{\mathbb{C}}

\newcommand{\D}{\mathbb{D}}

\newcommand{\cF}{{\mathcal{F}}}

\newcommand{\cA}{{\mathcal{A}}}

\newcommand{\Card}{\mathrm{Cardinal}}

\newcommand{\Hol}{\mathrm{Hol}}

\newtheorem{theorem}{\bf Theorem}
\newtheorem{cor}{\bf Corollary}

\newtheorem{thm}{\bf Theorem}[section]

\newtheorem*{prop*}{\bf Proposition}

\newtheorem{lem}[thm]{\bf Lemma}

\usepackage{enumitem}
						 \usepackage[%% pdftex,% might be luatex, just allow automatic default
						    colorlinks, hyperindex, plainpages=false, bookmarksopen, bookmarksnumbered, pdfusetitle, ]{hyperref}

\title[Complete Interpolating sequences for small Fock Spaces.] {Complete Interpolating sequences for small Fock Spaces.} \date{\today}
\author[Y. Omari]{Youssef OMARI}

\thanks{Research supported by CNRST Grant 84UM52016} 


\address{Laboratory of Mathematic Analysis and Applications, Faculty of Sciences, Mohammed V University in Rabat, 4 Av. Ibn Battouta, Morocco.}

\email{omariysf@gmail.com}
\begin{document}
\maketitle
\begin{abstract}
We give a characterization of complete interpolating sequences for the Fock spaces $\cF^p_\varphi,\ 1\leq p<\infty$, where $\varphi(z)=\alpha\left(\log^+|z|\right)^2,\ \alpha>0$. Our results are {analogous} to the classical Kadets-Ingham's $1/4-$Theorem on perturbation of Riesz bases of complex exponentials, and they answer a question asked by A. Baranov, A. Dumont, A. Hartmann and K. Kellay in \cite[page 31]{baranov2015sampling}.
 \end{abstract}

\section{Introduction and main results}\label{introduction}

Let $\alpha>0$ and $\varphi(z)=\alpha\left(\log^+|z|\right)^2,\ z\in\C$. The associated {\it Fock spaces} $\cF^p_\varphi$, $1\leq p<\infty$,  are the following
\begin{eqnarray}\nonumber
\cF^p_\varphi:=\left\{f\in\Hol(\C)\ :\ \|f\|^p_{p,\varphi}:=\int_\C \left|f(z)e^{-\varphi(z)}\right|^p\ dm(z)<\infty\right\}
\end{eqnarray}
and 
\begin{eqnarray}\nonumber
\cF^\infty_\varphi:=\left\{f\in\Hol(\C)\ :\ \|f\|_{\infty,\varphi}:= \sup_{z\in\C}\ \left|f(z)e^{-\varphi(z)}\right|\ <\infty\right\},
\end{eqnarray}
where $dm$ stands for the area Lebesgue measure in the complex plane $\C$. The Fock space $\cF^p_\varphi$ endowed with the above norm is a Banach space for every $1\leq p\leq \infty$.\\

{In order to define interpolating and sampling sets for $\cF^p_\varphi,\ 1\leq p<\infty$, we recall that  the point evaluation functional $L_z\ :\ f\in\cF^p_\varphi\longmapsto f(z)\in\C$, for fixed $z\in\C$, is a bounded linear map. Its norm $\|L_z\|_{\cF^p_\varphi\rightarrow\C}$ satisfies
\begin{eqnarray}
\frac{1}{C}(1+|z|)^{-2/p}e^{\varphi(z)}\ \leq\ \|L_z\|_{\cF^p_\varphi\rightarrow\C}\ \leq C(1+|z|)^{-2/p}e^{\varphi(z)},
\end{eqnarray}
for some constant $C\geq 1$, (see Lemma \ref{lem0}).}\\
 
 Let $\Lambda\subset\C$ be a countable set. The sequence $\Lambda$ is called {\it sampling} for $\cF^p_\varphi,$ $1\leq p<\infty,$ if there exists a constant $C\geq 1$ such that 
\begin{eqnarray}\nonumber
\frac{1}{C}\ \|f\|_{p,\varphi}^p\ \leq\ \|f\|_{p,\varphi,\Lambda}^p\ \leq C\ \|f\|_{p,\varphi}^p,\quad
\end{eqnarray}
for all $f\in\cF^p_\varphi$, where 
\begin{eqnarray*}
\|f\|_{p,\varphi,\Lambda}^p\ :=\ \sum_{\lambda\in\Lambda}\ \left(1+|\lambda|\right)^2\ |f(\lambda)|^p\ e^{-p\varphi(\lambda)}.
\end{eqnarray*}
Similarly, we say that a sequence $\Lambda$ is {\it sampling} for $\cF^\infty_\varphi$ whenever there exists a constant $C>0$ such that 
\begin{eqnarray}\nonumber
\|f\|_{\infty,\varphi}\ \leq\ C\ \|f\|_{\infty,\varphi,\Lambda},\ 
\end{eqnarray}
for every $f\in\cF^\infty_\varphi,$ where 
\[\|f\|_{\infty,\varphi,\Lambda}\ :=\ \sup_{\lambda\in\Lambda}\ |f(\lambda)|e^{-\varphi(\lambda)}.\] 
The set $\Lambda\subset\C$ is said to constitute an {\it interpolating sequence} for $\cF^p_\varphi,\ 1\leq p\leq \infty,$ if for every sequence $v=(v_\lambda)_{\lambda\in\Lambda}$ that satisfies $\|v\|^p_{p,\varphi,\Lambda}\ <\infty$,  
%\[\|v\|^p_{p,\varphi,\Lambda}\ :=\ \sum_{\lambda\in\Lambda}\ \left(1+|\lambda|\right)^2\ |v_\lambda|^p\ e^{-p\varphi(\lambda)}\ <\ \infty\]
there exists a function $f\in\cF^p_\varphi$ such that $f(\lambda)=v_\lambda$, for every $\lambda\in\Lambda$. It is called {\it a complete interpolating sequence} for $\cF^p_\varphi,\ 1\leq p\leq \infty$, whenever it is simultaneously sampling and interpolating for $\cF^p_\varphi$. Finally, for the Hilbert case $p=2$, standard arguments ensure that a sequence $\Lambda$ is a complete interpolating set for $\cF^2_\varphi$ if and only if the system of normalized reproducing kernels $\{\Bbbk_\lambda\}_{\lambda\in\Lambda}$ is a Riesz basis for $\cF^2_\varphi$, that means  $\{\Bbbk_\lambda\}_{\lambda\in\Lambda}$ is a linear isomorphic image of an orthonormal basis for $\cF^2_\varphi$.\\

%\textcolor{blue}{
The problem of existence of complete interpolating sequences for the Fock spaces has occupied many authors. Recall that K. Seip in \cite{seip1992density} and K.~Seip and R.~Wallst{\'e}n in \cite{seip1992densityy}, %he completely characterized sampling and interpolating sequences for the classical Fock spaces $(\varphi(z) = \alpha|z|^2,\ \alpha>0)$,  they 
proved that there exists no complete interpolating set for the classical Fock spaces $\cF^p$ $(\varphi(z) = \alpha|z|^2,\ \alpha>0)$, see also \cite{berndtsson1995interpolation}. The absence of such sequences was obtained by J.~Ortega-Cerd\`{a} and K. Seip in \cite{ortega1998beurling} for Fock spaces with weights $\varphi$ satisfying $\frac{1}{C}\leq \Delta \varphi(z)\leq C$, for some constant $C\geq 1$. Later on, for the subharmonic weights $\varphi(z)$ {whose associated} Riesz measures are doubling, the nonexistence of complete interpolating sets was established by N. Marco, X. Massaneda and J. Ortega-Cerd{\`a} in \cite{marco2003interpolating}. A. Borichev, R. Dhuez and K. Kellay in \cite{BORICHEV2007563} checked that large Fock spaces (those associated with weighted $\varphi(z)$ growing more rapidly than $|z|^2$ and satisfying some natural regularity conditions) have no complete interpolating set. {Furthermore, S.~Brekke and K.~Seip in  \cite{brekke1993density} and  A.~Borichev, A.~Hartmann, K.~Kellay, and X.~Massaneda in \cite{borichev2017geometric} showed that the spaces $\cF^p$ possess no multiple complete interpolating sequences.}\\

On the other hand, for $\varphi(z) = \mbox{Const}.\log|z|$, the associated Fock spaces are of finite dimension and obviously every sequence $\Lambda\subset\C$ of distinct points that satisfies $\Card(\Lambda)=\dim(\cF^p_\varphi)$ is a complete interpolating set for $\cF^p_\varphi$. In \cite{borichev_lyubarskii_2010}, A. Borichev and {Yu. Lyubarskii} provided Riesz bases of normalized reproducing kernels for $\cF^2_\varphi$, where $\varphi(z) = \left(\log^+|z|\right)^{\beta}; 1<\beta \leq 2$. They proved also that $\cF^2_\varphi$ possesses such bases if and only if $\varphi(x)$ grows at most as $\left(\log(x)\right)^2$. More recently, A. Baranov, {Yu. Belov} and A. Borichev in \cite{borichev2017fock} described the radial Hilbert Fock spaces which have Riesz bases of normalized reproducing kernels and are (or are not) isomorphic to de Branges spaces, see also \cite{baranov2015spectral}. In \cite{baranov2015sampling}, A. Baranov, A. Dumont, A. Hartmann and K. Kellay gave a complete description of  complete interpolating sequences for $\cF^p_\varphi$, when $p\in\{2, \infty\}$ and $\varphi(z) = \alpha \left(\log^+|z|\right)^2$. Riesz bases of normalized reproducing kernels for $\cF^2_\varphi$, where $\varphi=\left(\log^+|z|\right)^{\beta}; 1<\beta<2$, are characterized recently by K. Kellay and Y. Omari in \cite{kellay2018riesz}.\\ %{The results obtained in \cite{avdonin1974question} are in the spirit of  Avdonin's theorem for complex exponential systems for Paley-Wiener spaces.} \\ 
%\textcolor{blue}{They then asked for an analogue description of complete interpolating sequences for $\cF^p_\varphi$, when $p\notin\{2,\infty\}$. In this paper, we answer this question by giving a characterization of complete interpolating sequences for $\cF^p_\varphi$.}

The central result of this paper is a characterization of complete interpolating sequences for $\cF^p_\varphi$, $1\leq p<\infty\ \mbox{and}\ \varphi(z)\ =\ \alpha \left(\log^+|z|\right)^2$, where $\alpha>0$.  An interesting feature is that for our case the endpoint of the scale $p = 1$ also admits complete interpolating sequences. This happens because the discrete Hilbert transform in sparse sequences is bounded even for $p = 1$, see \cite{belov2011discrete} for more details. Our results answer a question asked by the authors in \cite{baranov2015sampling}. In fact, we employ the ideas and methods suggested in \cite{baranov2015sampling} and we develop additional techniques in order to cover {the remaining values} $1\leq p<\infty$. Finally, we directly treat the problem in the space $\cF^p_\varphi$ without using a complex interpolating theorem. \\

% % % % % % % % % % % % % % % % % % % % % % % % % % % % % % % % % % % % % % % % % % % % % % % % % % % % % % % % % %
% % % % % % % % % % % % % % % % % % % % % % % % % % % % % % % % % % % % % % % % % %
% % % % % % % % % % % % % % % % % % % % % % % % % % % % % % % % % % % % % % % % % % % % % %ù
% % % % % % % % % % % % % % % % % % % % % % % % % % % % % % % % % % % % % % % % % % % % % % % % % % % % % % % % % % % % % % % % % % % %
Before stating our main results, we give some necessary definitions and notations. In order to be more precise, let $\varphi(z)\ =\ \alpha\left(\log^+|z|\right)^2$, where $\alpha>0$, we set 
\[\rho(z)\ :=\ \left(\Delta\varphi(z)\right)^{-1/2} = |z|/\sqrt{2\alpha}.\]
We associate with the function $\rho$ the "distance" as follows :
\begin{eqnarray}
d_\rho(z,w)\ :=\ \frac{|z-w|}{1+\min\left\{\rho(z),\rho(w)\right\}}, \quad z,w\in \C.
\end{eqnarray}
Now, let $\Lambda$ be a sequence of complex numbers. We say that $\Lambda$ is {\it $d_\rho-$separated}, if there exists $\delta>0$, such that
\begin{eqnarray}
\inf\left\{d_\rho(\lambda,\lambda^\prime)\ :\ \lambda\neq\lambda^\prime,\ \lambda,\lambda^\prime\in\Lambda\right\}\ \geq\ \delta. \nonumber
\end{eqnarray}
{It is not difficult to see that a sequence $\Lambda$ is $d_\rho-$separated if and only if there exists a constant $c\in(0,1)$ such that the discs $\left\{\D(\lambda,c\rho(\lambda))\right\}$ are pairwise disjoint.} In what follows, for $\Gamma=\{\gamma_n \}$ a sequence of complex numbers ordered in such a way that $\left\{|\gamma_n|\right\}$ is nondecreasing, {we denote by $\Gamma\cup\{*\}$ the sequence consisting of $\Gamma$ union any point from $\C\setminus\Gamma$, and $\Gamma\setminus\{*\}$ the resulting sequence by removing any one point from $\Gamma$.}
%the notation $\Gamma \cup \{*\}$ stands for the set obtained from  $\Gamma$ by addition of any point not from $\Gamma$, and $\Gamma\setminus\{*\}$ the resulting sequence by removing any one point from $\Gamma$}.
 Also, for a given real sequence $\left(\delta_n\right)$ we use the following notations
\begin{equation*}
\nabla_N:=\inf_{n}\frac{1}{N}\sum_{k=n+1}^{n+N}\ \delta_k,\qquad  \Delta_N := \sup_n\frac{1}{N}\sum_{k=n+1}^{n+N}\ \delta_k,
\end{equation*}
where $N$ is a positive integer. For $1\leq p<\infty$, we denote by $q$ its H\"{o}lder conjugate exponent $(1/p+1/q=1)$. The main results of this work are the following theorems 

\begin{theorem}\label{thm3}
Let $\alpha>0,\ \varphi(r)=\alpha\left(\log^+ r\right)^2,\ \Lambda=\{\lambda_n=e^{\frac{1+n}{2\alpha}}\ :\ n\geq 0\}$ and  $\Gamma=\{\gamma_n\ :\ n\geq 0\}\subset\C$ such that $|\gamma_n|\leq |\gamma_{n+1}|$, we write $\gamma_n=\lambda_ne^{\delta_n}e^{i\theta_n}$, where $\delta_n,\theta_n\in\R$.  Then
\begin{enumerate}[label=$\Alph*)$,leftmargin=* ,parsep=0cm,itemsep=0cm,topsep=0cm]
\item\label{AA} $\Gamma\setminus\{*\}$ is a complete interpolating sequence for $\cF^p_\varphi,\ 1\leq p< 4/3$, if and only if the following three assertions hold :
\begin{enumerate}[label=$(\alph*)$,leftmargin=* ,parsep=0cm,itemsep=0cm,topsep=0cm]
\item\label{aa} $\Gamma$ is $d_\rho-$separated,
\item\label{bb} $\left(\delta_n\right)$ is a bounded real sequence,
\item\label{cc} there exists an integer $N\geq 1$ such that
\begin{eqnarray}
-\left(\frac{1}{4\alpha}+ \frac{1}{q\alpha}\right)\ <\ \nabla_N\ \leq\ \Delta_N\ <\ \frac{1}{4\alpha}-\frac{1}{q\alpha}.\label{(A)}
\end{eqnarray}
\end{enumerate}
\item\label{BB} $\Gamma$ is a complete interpolating set for $\cF^p_\varphi,\ 4/3<p<4$, if and only if \ref{aa} and \ref{bb} are verified and 
\begin{eqnarray}\label{delta}
\frac{1}{4\alpha}- \frac{1}{q\alpha}\ <\ \nabla_N\ \leq\ \Delta_N\ <\ \frac{1}{p\alpha}-\frac{1}{4\alpha},\label{(B)}
\end{eqnarray}
for some integer $N\geq 1$.\\
\item\label{CC} $\Gamma\cup\{*\}$ is a complete interpolating sequence for $\cF^p_\varphi,\ p>4,$ if and only if \ref{aa} and \ref{bb} are satisfied and for some positive integer $N$ we have
\begin{eqnarray}
\frac{1}{p\alpha}-\frac{1}{4\alpha}\ < \nabla_N\ \leq\ \Delta_N\ <\ \frac{1}{p\alpha}+\frac{1}{4\alpha}.\label{(C)}
\end{eqnarray}
\end{enumerate}
\end{theorem}
For $p\in\{4/3,\ 4\}$, we have the following
\begin{theorem}\label{thm2}
Let $\alpha>0,\ \varphi(r)=\alpha\left(\log^+ r\right)^2$ and $\Lambda=\{\lambda_n=e^{\frac{1+n}{2\alpha}}\ :\ n\geq 0\}$. Let $\Gamma=\{\gamma_n\ :\ n\geq 0\}\subset\C$ such that $|\gamma_n|\leq|\gamma_{n+1}|$, we write $\gamma_n=\lambda_ne^{\delta_n}e^{i\theta_n}$, where $\delta_n,\theta_n\in\R$. The sequences $\left(e^{-\frac{1}{4\alpha}}\Gamma\right)\setminus\{*\}$ and $e^{-\frac{1}{4\alpha}}\Gamma$ are complete interpolating sets for $\cF^{4/3}_\varphi$ and $\cF^{4}_\varphi$ respectively, if and only if the conditions \ref{aa} and \ref{bb} hold and  
\begin{equation}
-\frac{1}{4\alpha}\ <\ \nabla_N\ \leq\ \Delta_N\ <\ \frac{1}{4\alpha}
\end{equation}
 for some integer $N\geq 1$.
\end{theorem}

% % % % % % % % % % % % % % % % % % % % % % % % % % % % % % % % % % % % % % % % % % % % % % % % % % % % % % % % % %
Note that the cases $p=2, \infty$ were treated in \cite{baranov2015sampling}. If $N=1,\ p=2$, the condition \eqref{delta} in Theorem \ref{thm3} becomes $\sup_{n\geq 0}\ \left|\delta_n\right|< 1/(4\alpha)$, which is analogue to the well known Kadets-Ingham's $1/4-$Theorem, see \cite{kadets1964exact,seip2004interpolating}. This latter is related to a stability problem of Riesz bases of normalized reproducing kernels for the Paley-Weiner spaces $PW^2_\alpha$, for more details on this problem we refer to \cite{hruvsvcev1981unconditional,lyubarskii2002weighted} and references therein. Also, for an arbitrary integer $N\geq 1$, Theorem \ref{thm3} {appears like Avdonin's Theorem} (see, e.g. \cite{avdonin1974question,baranov2015sampling}). For $1< p<\infty$, Theorems \ref{thm3} and \ref{thm2} give similar results to those proved by K. Seip and {Yu. Lyubarskii} in \cite{Lyubarskii97completeinterpolating} about complete interpolating sets for $PW^p_\alpha,\ 1<p<\infty,$ and those proved by J. Marzo and K. Seip in \cite{marzo2009kadets} for the space of polynomials equipped with the $\Ll^p-$norm. \\ %Contrary to these works, we have some critical cases as shown in Theorem \ref{thm3}.\\

% % % % % % % % % % % % % % % % % % % % % % % % % % % % % % % % % % % % % % % % % % % % % % % % % % % % % % % % % %
% % % % % % % % % % % % % % % % % % % % % % % % % % % % % % % % % % % % % % % % % %
% % % % % % % % % %
We can summarize Theorem \ref{thm3} and Theorem \ref{thm2} by the following {table}, let $\Gamma=\{\gamma_n\}_{n\geq 0}$ be a sequence of $\C$ and write $\gamma_n=\lambda_ne^{\delta_n}e^{i\theta_n}$, we have\\
%\begin{table}[b]\vspace*{-3ex}
%\caption[]{A small table.} \label{mytable}
\begin{center}
{\renewcommand{\arraystretch}{2} %donne la distance entre les lignes%
{\setlength{\tabcolsep}{0.1cm} %donne la distance entre les collones%
\begin{tabular}{|c|c|c|}
  \hline
  $p$ & Complete interpolating sequence  & $\nabla_N\leq\Delta_N$ \\
%      & sequence & \\
  \hline
 $\displaystyle 1\leq p< \frac{4}{3}$  &  $\Gamma\setminus\{*\}$&  $\displaystyle -\left(\frac{1}{4\alpha}+ \frac{1}{q\alpha}\right)\ <\ \nabla_N\ \leq\ \Delta_N\ <\ \frac{1}{4\alpha}-\frac{1}{q\alpha}$ \\
   \hline
 $\displaystyle p=\frac{4}{3}$ &$\displaystyle\left(e^{-\frac{1}{4\alpha}}\Gamma\right)\setminus\{*\}$  & $\displaystyle-\frac{1}{4\alpha}<\nabla_N\leq\Delta_N<\frac{1}{4\alpha}$ \\
   \hline
$\displaystyle\frac{4}{3}<p<4$&    $\Gamma$& $\displaystyle\frac{1}{4\alpha}-\frac{1}{q\alpha}<\nabla_N\leq\Delta_N<\frac{1}{p\alpha}-\frac{1}{4\alpha}$ \\
  \hline
 $p=4$& $\displaystyle e^{-\frac{1}{4\alpha}}\Gamma$ & $\displaystyle-\frac{1}{4\alpha}<\nabla_N\leq\Delta_N<\frac{1}{4\alpha}$  \\
   \hline
$p>4$&   $\Gamma\cup\{*\}$ & $\displaystyle\frac{1}{p\alpha}-\frac{1}{4\alpha}<\nabla_N\leq\Delta_N<\frac{1}{4\alpha}+\frac{1}{p\alpha}$  \\
  \hline
\end{tabular}}}
\end{center}
\vspace*{0.5cm}

 % % % % % % % % % % % % % % % % % % % % % % % % % % % % % % % % % % % %ù
% % % % % % % % % % % % % % % % % % % % % % % % % % % % % % % % % % % % % % % % % % % % % % % % % % % % % % % % % % % % % % % % % % % %
The paper is organized as follows : In the next section, we prove some technical lemmas and preliminary results needed in the proofs of the main theorems. Section \ref{proof-thm1} is devoted to prove Theorem \ref{thm1} (Section \ref{key-lemmas}). Theorems \ref{thm3} and \ref{thm2} will be proved in the last section. \\

We end this introduction with some words on notation. Throughout this paper, the notation $U(z)\lesssim V(z)$ for $z$ in some set $\Omega$ means that the ratio $U(z)/V(z)$ of the two positive functions $U(z)$ and $V(z)$ is bounded from above by a positive constant independent of $z$ in $\Omega$. We write $U(z)\asymp V(z)$ if both $U(z)\lesssim V(z)$ and $V(z)\lesssim U(z)$ hold simultaneously.
% % % % % % % % % % % % % % % % % % % % % % % % % % % % % % % % % % % % % % % % % % % % % % % % % % % % % % % % % %
% % % % % % % % % % % % % % % % % % % % % % % % % % % % % % % % % % % % % % % % % %
% % % % % % % % % % % % % % % % % % % % % % % % % % % % % % % % % % % % % % % % % % % % % %ù
% % % % % % % % % % % % % % % % % % % % % % % % % % % % % % % % % % % % % % % % % % % % % % % % % % % % % % % % % % % % % % % % % % % %
\section{Key Lemmas and preliminary results}\label{key-lemmas}
% % % % % % % % % % % % % % % % % % % % % % % % % % % % % % % % % % % % % % % % % % % % % % % % % % % % % % % % % %
% % % % % % % % % % % % % % % % % % % % % % % % % % % % % % % % % % % % % % % % % %
% % % % % % % % % % % % % % % % % % % % % % % % % % % % % % % % % % % % % % % % % % % % % %ù
% % % % % % % % % % % % % % % % % % % % % % % % % % % % % % % % % % % % % % % % % % % % % % % % % % % % % % % % % % % % % % % % % % % % \subsection{Key Lemmas.}
In this section, we will prove some key lemmas and secondary results. In fact, we give the necessary estimates of some infinite products and we also prove some preliminary results about $d_\rho-$separated sequences. Throughout the rest of this paper, we assume that $\alpha=1$ and consequently $\varphi(r)=\left(\log^+ r\right)^2$. {We begin by estimating the norm of the point evaluation functional.

\begin{lem}\label{lem0}
Let $z\in\C$ be fixed. The point evaluation 
functional $L_z : f\in\cF^p_\varphi \longmapsto f(z)\in\C$ is a bounded linear map and its norm satisfies the following estimate
\[\|L_z\|_{\cF^p_\varphi\rightarrow\C}\ \asymp\ (1+|z|)^{-2/p}e^{\varphi(z)}.\] 
\end{lem}
\begin{proof}
Let $f\in \cF^p_\varphi$ and set $F(r)=\int_{0}^{2\pi}|f(re^{i\theta})|^p\frac{d\theta}{2\pi}.$ First we write
\begin{eqnarray*}
\|f\|^p_{p,\varphi} %& = & \int_\C |f(z)|^pe^{-p\varphi(z)}dm(z) \\
        %& = & \int_1^\infty \left(\int_{0}^{2\pi}  |f(re^{i\theta})|^p \frac{d\theta}{2\pi} \right)e^{-p\varphi(r)} rdr \\
         \asymp  \int_1^\infty F(r) e^{-p\varphi(r)} rdr
       % & = & \int_{1}^{\infty} e^{p\omega(\log r)-p(\log r)^2+2\log r} \frac{dr}{r}\\
        & = & \int_{0}^{\infty} F(e^t) e^{-pt^2+2t} dt.
\end{eqnarray*}
Let $z\in\C$ and write $u:=\log|z|$. Let $m\in\mathbb N$ be the integer that satisfies $m\leq u<m+1$ and $0<\varepsilon<1$ sufficiently small, we have
\begin{eqnarray*}
\|f\|_{p,\varphi}^p & \gtrsim & \int_{u-\varepsilon}^{u+\varepsilon} F(e^t) e^{-pt^2+2t} dt\\
       % & = & \int_{-\varepsilon}^{\varepsilon} e^{p\omega(t+u)-pu^2-2ptu-pt^2+2t+2u} dt\\
        %& \asymp & e^{-pu^2+2u+2mpu}\int_{-\varepsilon}^{\varepsilon} e^{p\omega(t+u)-2pm(t+u)} dt\\
        & \asymp & e^{-pu^2+2u+2pmu}\int_{u-\varepsilon}^{{u+\varepsilon}} F(e^t)e^{-2pmt}\ dt\\
        & = &  e^{-pu^2+2u+2pmu}\int_{e^{u-\varepsilon}}^{e^{u+\varepsilon}} F(t)e^{-2pm \log t} \frac{dt}{t}\\
        & = & e^{-pu^2+2u+2pmu}\int_{A(z,e^{-\varepsilon},e^\varepsilon)}\  |f(\xi)|^p/|\xi|^{2pm+2}\ dm(\xi),
\end{eqnarray*}
where $A(z,e^{-\varepsilon},e^\varepsilon):=\{\xi\in\C \ :\ |z|e^{-\varepsilon}\leq |\xi|\leq|z|e^\varepsilon\}$. Applying now mean's theorem to the function $|f(\xi)|^p/|\xi|^{2pm+2}$ in a disk $\D(z,\delta|z|)$, for some $\delta>0$ small enough so that $\D(z,\delta|z|)\subset A(z,e^{-\varepsilon},e^\varepsilon)$. We obtain
\begin{eqnarray*}
|f(z)|^p/|z|^{2pm+2} & \leq & \frac{C}{|z|^2} \int_{\D(z,\delta|z|)} |f(\xi)|^p/|\xi|^{2pm+2} dm(\xi) \\
     & \leq & \frac{C}{|z|^2}\int_{A(z,e^{-\varepsilon},e^\varepsilon)}\ |f(\xi)|^p/|\xi|^{2mp+2} dm(\xi).
\end{eqnarray*}
Thus,
\begin{eqnarray*}
|f(z)|^p e^{-pu^2+2u} & \lesssim & e^{-pu^2+2u+2mpu}\int_{A(z,e^{-\varepsilon},e^\varepsilon)}\ |f(\xi)|^p/|\xi|^{2pm+2}dm(\xi)\ \lesssim \|f\|_{p,\varphi}^p.
\end{eqnarray*}
%According to \cite[Lemma 4.1]{BORICHEV2007563}, there exists a constant $C=C(\varphi,p)$ such that 
%\[\left|f(z)\right|^pe^{-p\varphi(z)}\ \leq \frac{C}{(1+|z|)^2}\int_{\D(z,\frac{1+|z|}{2})}\ \left|f(\xi)\right|^pe^{-p\varphi(\xi)} dm(\xi)\leq C\frac{\|f\|^p_{p,\varphi}}{(1+|z|)^2},\quad f\in\cF^p_\varphi .\]
%Hence, $\|L_z\|_{\cF^p_\varphi\rightarrow\C}\leq C(1+|z|)^{-2/p}e^{\varphi(z)}$.
For the reverse estimate, let $u=\log|z|$ and let $n$ be the integer that satisfies $n\leq 2u<n+1$. Consider next the function $f(\xi)=\xi^n$. We have
\begin{eqnarray}
\|f\|_{p,\alpha}^p & = & \int_\C |\xi|^{pn}e^{-p\varphi(\xi)}dm(\xi)\ \asymp\ \int_{1}^{\infty} t^{pn}e^{-p(\log t)^2} tdt \nonumber\\
   & = & \int_{0}^{\infty} e^{(pn+2)t-pt^2} dt\ =\ e^{\frac{p}{4}\left(n+2/p\right)^2}\int_{-(n+2/p)}^{\infty} \ e^{-pt^2} dt \nonumber\\
   & \lesssim & e^{(pn+2)u-pu^2}.
\end{eqnarray}
Since $|f(z)|=e^{nu}$, we  obtain $\|L_z\|_{\cF^p_\varphi\rightarrow\C}\ \gtrsim\ e^{u^2-\frac{2}{p}u}$ and this completes the proof.
\end{proof}
}
Now, for a given sequence $\Lambda$ of $\C$, $G_\Lambda$  is the following infinite product, whenever it converges, associated with $\Lambda$
\begin{eqnarray}
G_\Lambda(z) := \prod_{\lambda\in\Lambda}\left(1-\frac{z}{\lambda}\right),\quad z\in\C.
\end{eqnarray}
In what follows, we consider the sequences
\begin{eqnarray}
\Lambda_{l}:=\left\{\lambda_n:=e^{\frac{1+n}{2}}{e^{i\theta_n}}\ :\ n\geq l \right\},
\end{eqnarray}
where $l \in \{-1,0,1\}$ and $\theta_n$ are arbitrary real numbers. The notation $\dist(z,\Lambda)$ stands for the Euclidean distance between $z$ and  $\Lambda$. The following lemma describes some basic properties of the infinite product associated with these sequences.
\begin{lem}\label{lem1}
Let $l \in \{-1,0,1\}$,  the infinite product $G_{\Lambda_{l}}$ converges uniformly on compact sets of $\C$ and hence it defines an entire function. We have the following estimates
\begin{eqnarray}\label{est1}
\left|G_{\Lambda_{l}}(z)\right| e^{-\varphi(z)}\ \asymp\ \frac{\dist(z,\Lambda_{l})}{1+|z|^{3/2+l}},\qquad z\in\C, \label{estim}
\end{eqnarray}
and
\begin{eqnarray}\label{est2}
\left|G'_{\Lambda_{l}}(\lambda)\right|\ e^{-\varphi(\lambda)}\ \asymp\ \frac{1}{|\lambda|^{3/2+l}},\qquad \lambda\in \Lambda_{l}.
\end{eqnarray}
\end{lem}

\begin{proof}
{The estimate \eqref{est2} can be obtained from  \eqref{est1} by letting tend $z$ to $\lambda$. For the proof of  \eqref{est1} we refer to \cite[Lemma 2.6]{borichev_lyubarskii_2010}.}
\end{proof} 

The following lemma can be {obtained} easily from the proof of \cite[Theorem 1.1]{baranov2015sampling}. We include the proof for completeness.

\begin{lem}\label{estimate}
Let $l \in \{-1,0,1\}$, $N\geq 1$, and $\Lambda_{l}=\left\{\lambda_n\ :\ n\geq l \right\}$. Let $\Gamma_l=\{\gamma_n\ :\ n\geq l\}$ and write $\gamma_n=|\lambda_n| e^{\delta_n}e^{i\theta_n}$. Assume that the conditions \ref{aa} and \ref{bb} are satisfied, then the infinite product $G_{\Gamma_l}$ satisfies the estimate
\begin{eqnarray}\nonumber
\frac{\dist(z,\Gamma_l)}{1+|z|^{3/2+l+2\delta}}\ \lesssim\ \left|G_{\Gamma_l}(z)\right|e^{-\varphi(z)}\ \lesssim\ \frac{\dist(z,\Gamma_l)}{1+|z|^{3/2+l+2\eta}},\qquad z\in\C,
\end{eqnarray}
where $\delta=\Delta_N$ and $\eta=\nabla_N$. {Also, for every $\gamma_n\in\Gamma$ we have
\begin{eqnarray}\nonumber
\frac{e^{\varphi(\gamma_n)}}{1+|\gamma_n|^{3/2+l+2\delta}}\ \lesssim\ \left|G'_{\Gamma_l}(\gamma_n)\right|\  \lesssim\ \frac{e^{\varphi(\gamma_n)}}{1+|\gamma_n|^{3/2+l+2\eta}}.
\end{eqnarray}
}
\end{lem}

\begin{proof}
Let $|z|=e^{t}$, such that $|\gamma_{n-1}|\leq |z|\leq |\gamma_n|$, and assume that $\dist(z,\Gamma)=\left|z-\gamma_{n-1}\right|$. Take $m$ the integer number that satisfies $$\frac{m}{2}-\frac{1}{4}\leq t< \frac{m}{2}+\frac{1}{4}.$$
First, remark that $|m-n|$ is uniformly bounded in $|z|$. Now, we have 
\begin{eqnarray}
\log\left|G_{\Gamma_l}(z)\right| & = & \sum_{l\leq k< n-1}\ \log\left|\frac{z}{\gamma_k}\right|+\log\left|1-\frac{z}{\gamma_{n-1}}\right| + O(1) \nonumber \\
  & = & \sum_{k=l}^{n-1}\ \log\left|\frac{z}{\gamma_k}\right|+\log\frac{\dist(z,\Gamma_l)}{|z|} + O(1) \nonumber \\ 
  & = & \sum_{k=l}^{m-1}\ \left(\log|z|-\frac{k+1}{2}-\delta_k\right)+\log\frac{\dist(z,\Gamma_l)}{|z|} + O(1) \nonumber \\ 
  & = & (m-l)\log|z|-\frac{m(m+1)}{4}-\sum_{k=0}^{m-1}\ \delta_k +\log\frac{\dist(z,\Gamma_l)}{|z|} + O(1). \label{A}
\end{eqnarray}
For a fixed $N\geq 1$, we write $m=sN+r$, for some $0\leq r<N$. On the one hand, we get
\begin{eqnarray}
\sum_{k=0}^{m-1}\ \delta_k  =  \sum_{k=0}^{s-1}\sum_{j=kN}^{kN+N-1}\ \delta_j + O(1)  \leq  \Delta_N sN + O(1) = 2\Delta_N \log|z| + O(1). \label{B}
\end{eqnarray}
On the other hand we have
\begin{eqnarray}
\sum_{k=0}^{m-1}\ \delta_k =  \sum_{k=0}^{s-1}\sum_{j=kN}^{kN+N-1}\ \delta_j + O(1)  \geq \nabla_N sN + O(1) = 2\nabla_N \log|z| + O(1). \label{C}
\end{eqnarray}
The result follows immediately by combining \eqref{A}, \eqref{B} and \eqref{C}. 
\end{proof}

We will need the following classical lemma for which we include the proof for completeness.

\begin{lem}\label{separation}
Let $\nu$ be a real parameter and $0<p<\infty$. Let $\Lambda\subset\C$ be a finite union of $d_\rho-$separated sequences. The integral
\begin{eqnarray}\label{integral}
\int_\C \left(\frac{\dist(z,\Lambda)}{1+|z|^{3/2+\nu}}\right)^p dm(z)
\end{eqnarray}
converges if and only if $\nu>2/p-1/2$.
\end{lem}

\begin{proof}
{Firstly, for every $z\in \C$ we have $\dist(z,\Lambda)\lesssim 1+|z|$. {On the other hand}, since the sequence $\Lambda$ is a finite union of $d_\rho-$separated sequences then the estimate $\dist(z,\Lambda)\gtrsim 1+|z|$ holds for every $z\in\{w\in\C\ :\ d_\rho(w,\Lambda)\geq \delta\}$, for some small positive real number $\delta$ and where
$$d_\rho(w,\Lambda)=\inf\left\{d_\rho(w,\lambda)\ :\ \lambda\in\Lambda \right\}.$$
The remaining part of the integral can be obtained by a simple argument of subharmonicity of the function $z\ \mapsto\frac{1}{(1+|z|)^{1/2+\nu}}$. Therefore, the convergence of the integral in \eqref{integral} is equivalent to the convergence of the integral 
\[\int_\C \left(\frac{1}{(1+|z|)^{1/2+\nu}}\right)^p dm(z). \]
This ends the proof.}

%\begin{eqnarray} \int_\C \left(\frac{\dist(z,\Lambda)}{1+|z|^{3/2+\nu}}\right)^p dm(z)   & \lesssim & \int_{\C} \left(\frac{1}{1+|z|^{1/2+\nu}}\right)^p\ dm(z). \end{eqnarray}
%{\color{blue}On the other hand}, the sequence $\Lambda$ is a finite union of $d_\rho-$separated sequences, then there exists $\delta>0$ such that every crown $\cC_{\lambda,\delta}:=\{z\in\C\ :\ |\lambda|e^{-\delta}<|z|\leq |\lambda|e^{\delta} \}$ contains a finite number of $\lambda'\in\Lambda\ (\Card(\Lambda\cap\cC_{\lambda,\delta})\lesssim 1)$. We denote $\cC_\Lambda := \underset{\lambda}{\bigcup}\  \cC_{\lambda,\delta}$, we write
%\begin{eqnarray}
%\int_\C \left(\frac{\dist(z,\Lambda)}{1+|z|^{3/2+\nu}}\right)^p dm(z) & = & \int_{\C\setminus\cC_\Lambda} + \int_{\cC_\Lambda} \left(\frac{\dist(z,\Lambda)}{1+|z|^{3/2+\nu}}\right)^p dm(z)   \nonumber \\
% & \geq & \int_{\C\setminus\cC_\Lambda} \left(\frac{\dist(z,\Lambda)}{1+|z|^{3/2+\nu}}\right)^p dm(z) \nonumber \\
% & \asymp & \int_{\C\setminus\cC_\Lambda} \left(\frac{1}{1+|z|^{1/2+\nu}}\right)^p dm(z). \end{eqnarray}
%However,  by an argument of sub-harmonicity we get
%\[\int_{\cC_\Lambda} \left( \frac{1}{|z|^{1/2+\nu}}\right)^p dm(z) \lesssim \int_{\C\setminus\cC_\Lambda} \left(\frac{1}{|z|^{1/2+\nu}}\right)^p dm(z).\]
%Consequently,
%\begin{align} \int_{\C}\ \left(\frac{\dist(z,\Lambda)}{1+|z|^{3/2+\nu}}\right)^p\ dm(z) & \gtrsim & \int_{\C} \left(\frac{1}{1+|z|^{1/2+\nu}}\right)^p\ dm(z). \end{eqnarray}
\end{proof}

A countable set $\Lambda$ of the complex plane is called \textit{a uniqueness set} for $\cF^p_\varphi$ whenever every function in $\cF^p_\varphi$ vanishing on $\Lambda$ is identically zero. $\Lambda$ is said to be  \textit{a zero set} for $\cF^p_\varphi$ if there exists $f\in\cF^p_\varphi\setminus\{0\}$  such that $\Lambda$ is exactly the zero set of $f$, counting multiplicities. We say that $\Lambda$ is \textit{a set of uniqueness of zero excess} or \textit{exact uniqueness set} for $\cF^p_\varphi$ whenever it is a set of uniqueness for $\cF^p_\varphi$ and when we remove any point of $\Lambda$ we obtain a zero set for $\cF^p_\varphi$. A direct consequence of the above lemma is the following corollary
\begin{cor}\label{coro1}
Let $\Lambda_l$ be the sequence defined as above. We have 
\begin{enumerate}[label=$(\arabic*)$,leftmargin=* ,parsep=0cm,itemsep=0cm,topsep=0cm]
    \item\label{one} The sequence $\Lambda_{-1}$ is a uniqueness set of zero excess for $\cF^p_\varphi$ if and only if $p\in (4,\infty]$.
 \item\label{two}
    The set $\Lambda_0$ is a uniqueness sequence of zero excess for $\cF^p_\varphi$ if and only if $4/3<p<4.$
    \item\label{three} Let $1\leq p<4/3$. The sequence $\Lambda_{1}$ is a uniqueness set of zero excess for $\cF^p_\varphi$.
\end{enumerate}
\end{cor}

\begin{proof}
$(1)$ Assume that there exists $f\in \cF^p_\varphi$ vanishing on $\Lambda_{-1}$, {taking into account} Hadamard's factorization theorem, we can write $f=hG_{\Lambda_{-1}}$, for an entire function $h$. By Lemma \ref{lem1} we have
\begin{eqnarray}
|f(z)|\ =\ \left|h(z)G_{\Lambda_{-1}}(z)\right|\ \asymp\ \frac{\dist(z,\Lambda_{-1})}{(1+|z|)^{1/2}}\ |h(z)| e^{\varphi(z)},\quad z\in\C. \label{estim-coro}
\end{eqnarray}
The last estimate and the fact that $f\in\cF^p_\varphi$ imply with Lemma \ref{lem0}
\begin{eqnarray}
|h(z)|\dist(z,\Lambda_{-1}) \ \lesssim\ (1+|z|)^{1/2-2/p},\quad z\in\C.
\end{eqnarray}
Using the fact that $\dist(z,\Lambda_{-1})\asymp 1+|z|$ for every $z\in  \cA:=\left\{w\in\C\ :\  \dist(w,\Lambda_{-1})\geq \delta|w| \right\}$, for some $\delta\in(0,1)$, this implies that $|h(z)|\lesssim(1+|z|)^{-1/2-2/p}$, for every $z\in \cA$ and therefore $h=0$. % we conclude that $h$ must be a polynomial, denote $k$ its degree. The identity \eqref{estim-coro} yields
%\begin{eqnarray*}
%\int_\C \left(\frac{\dist(z,\Lambda_{-1})}{(1+|z|)^{1/2-k}}\ \right)^p\ dm(z) \asymp \int_{\C}\ \left(\frac{\dist(z,\Lambda_{-1})}{(1+|z|)^{1/2}}\ |h(z)|\right)^p\ dm(z) <\infty.
%\end{eqnarray*}
%Thanks to Lemma \ref{separation}, we achieve that $h$ is the zero function,} 
Hence $\Lambda_{-1}$ is a uniqueness set for $\cF^p_\varphi$, for every $p\geq1$. Now, fix $\lambda\in\Lambda_{-1}$, the sequence $\Lambda'=\Lambda_{-1}\setminus\{\lambda\}$ is a zero set for $\cF^p_\varphi$, {whenever $p>4$}. Indeed, by Lemmas \ref{lem1} and \ref{separation} we have
\begin{eqnarray*}
\int_\C \left|\frac{G_{\Lambda_{-1}}(z)}{z-\lambda}\ e^{-\varphi(z)}\right|^pdm(z)\ \asymp\ \int_\C \left(\frac{\dist(z,\Lambda_{-1})}{(1+|z|)^{3/2}}\right)^p dm(z) <\infty.
\end{eqnarray*}  
Thus, the function $\frac{G_{\Lambda_{-1}}(z)}{z-\lambda}$ belongs to $\cF^p_\varphi$ and vanishes exactly on $\Lambda\setminus\{\lambda\}.$ 

{To prove the converse, it is sufficient to show that $\Lambda'=\Lambda_{-1}\setminus\{\lambda\}$ is a uniqueness set for $\cF^p_\varphi$, for every $p\leq 4$. Let $f\in\cF^p_\varphi$ be a given function that vanishes in $\Lambda'$.  Hadamard's factorization theorem ensures the existence of an entire function $h$ such that $f=h\frac{G_{\Lambda_{-1}}}{z-\lambda}$. Lemma \ref{lem1} implies that
\begin{eqnarray}
|f(z)|\ =\ \left|h(z)\frac{G_{\Lambda_{-1}}(z)}{z-\lambda}\right|\ \asymp\ \frac{\dist(z,\Lambda')}{(1+|z|)^{3/2}}\ |h(z)| e^{\varphi(z)},\quad z\in\C. \label{estim-cor}
\end{eqnarray}
Since $f\in\cF^p_\varphi$, we obtain 
\begin{eqnarray*}
 |h(z)| \dist(z,\Lambda')\ \lesssim (1+|z|)^{3/2-2/p},\quad z\in\C.
\end{eqnarray*}
The same arguments as above imply that  $h$ is a polynomial of degree at most $3/2-2/p-1=1/2-2/p\leq 0$, i.e. $h$ is a constant if $p=4$ and the zero function if $1\leq p<4$. Thus $\Lambda'$ is a uniqueness set for $\cF^p_\varphi$, for every $1\leq p<4$. For $p=4$, since  $f=c\frac{G_{\Lambda_{-1}}}{z-\lambda}\in\cF^p_\varphi$, the identity \eqref{estim-cor} implies that 
\begin{eqnarray}
\int_\C\left(\frac{\dist(z,\Lambda')}{(1+|z|)^{3/2}}\right)^pdm(z)\ \asymp\   \int_\C|f(z)|^pe^{-p\varphi(z)}dm(z) <\infty.
\end{eqnarray}
This contradicts Lemma \ref{separation} and hence $\Lambda'$ is a uniqueness set for $\cF^p_\varphi$, for every $p\leq4$.} This completes the proof.\\

%Following the same steps of the proof of \ref{one} one can easily obtain 
%$(2)$ %{\color{green}It was proved in the last proof that $\Lambda_{-1}\setminus\{\lambda\}$} 
The proofs of \ref{two} and \ref{three} in the corollary are completely the same as that of \ref{one}, we use the estimates of the functions $G_{\Lambda_l}$ in Lemma \ref{lem1}. The proof is complete.
\end{proof}
The following lemma is a crucial tool in the proof of our results

\begin{lem}\label{lem4}
Let $\Lambda$ be a sequence of complex numbers and let $1\leq p<\infty$. The following are equivalent:
\begin{enumerate}[label=$(\arabic*)$,leftmargin=* ,parsep=0cm,itemsep=0cm,topsep=0cm]
\item\label{T_Lam} There exists a constant $C>0$ such that
\begin{eqnarray}
 \sum_{\lambda\in\Lambda}\ (1+|\lambda|)^{2}\left|f(\lambda) e^{-\varphi(\lambda)}\right|^p\ \leq\ C \|f\|^p_{p,\varphi},
\end{eqnarray}
for every $f\in\cF^p_\varphi$.
\item  $\Lambda$ is a finite union of $d_\rho-$separated sequences.
\end{enumerate}
\end{lem}
\begin{proof}
The proof is similar to the classical one given for  \cite[Lemma 2.6]{baranov2015sampling}. %, therefore we omit it here.
\end{proof}
% % % % % % % % % % % % % % % % % % % % % % % % % % % % % % % % % % % % % % % % % % % % % % % % % % % % % % % % % % 
The description of complete interpolating sequences for $\cF^p_\varphi$ is obtained by comparing them to the sequence $\Lambda_l=\left\{e^{\frac{n+1}{2}}{e^{i\theta_n}}\right\}_{n\geq l}$, this leads us to the following result which will play an important role in the proof of the main theorems. The proof will be given in the next section
  %The result as follows will play an important role in the proof of the main results and it will be proved in the next section.
\begin{thm}\label{thm1}
 We have the following 
\begin{enumerate}[label=$(\arabic*)$,leftmargin=* ,parsep=0cm,itemsep=0cm,topsep=0cm]
    \item\label{1-}  The sequences $\Lambda_{-1},\ \Lambda_0$ and $\Lambda_1$ are complete interpolating sets for $\cF^p_\varphi$, for every $4<p<\infty,\ \frac{4}{3}< p< 4$ and $1\leq p < \frac{4}{3}$ respectively.\\

    \item\label{4-} The sets $e^{-\frac{1}{4}}\Lambda_{1}$ and $e^{-\frac{1}{4}}\Lambda_{0}$ are complete interpolating sequences for $\cF^p_\varphi$, when $p=\frac{4}{3}$ and $p=4$ respectively.
\end{enumerate}
\end{thm}

As an immediate consequence of Theorem \ref{thm1}, the quantity
\begin{eqnarray}
\|f\|_{p,\varphi,\Lambda}\ :=\ \left(\sum_{\lambda\in\Lambda}\ |\lambda|^2|f(\lambda)|^pe^{-p\varphi(\lambda)}\right)^{1/p},\quad f\in\cF^p_\varphi,
\end{eqnarray}
where $\Lambda$ is the corresponding sequence to $\cF^p_\varphi$, defines a norm equivalent to that of $\cF^p_\varphi$.\\

To make the proof of Theorem \ref{thm1} clearer, we single out the next technical ingredient as a lemma.
%We will need the next classical lemma in the proof of Theorem \ref{thm1}.
\begin{lem}\label{integ}
Let $\Lambda=\Lambda_l$, for some $l\in\{-1,0,1\}$ and $\lambda\in\Lambda$. Let $\beta>2$ be fixed, we have
\begin{eqnarray}\label{inte}
I_{\beta,\lambda} := \int_\C \frac{\dist(z,\Lambda)}{|z-\lambda|(1+|z|)^{\beta}}\ dm(z)\ \asymp\ |\lambda|^{2-\beta}.
\end{eqnarray}
Also if $\nu<1$, we have
\begin{equation}\label{prb}
 (1+|z|)^{-\nu}\ \sum_{\lambda\in\Lambda}\ |\lambda|^{\nu}\ \frac{\dist(z,\Lambda)}{|z-\lambda|}\ \asymp\ 1.
\end{equation}
\end{lem}
\begin{proof}
First, we write
\begin{eqnarray}
I_{\beta,\lambda}  & = &  \underbrace{\int_{|\lambda|/2\leq |z|\leq 2|\lambda|} \frac{\dist(z,\Lambda)}{|z-\lambda|(1+|z|)^{\beta}}\ dm(z)}_{I_1}+\underbrace{\int_{\C\setminus\{|\lambda|/2\leq |z|\leq 2|\lambda| \}}\ldots}_{I_2} \nonumber
      %      & = &  I_1\ +\ I_2\ . \nonumber
\end{eqnarray}
For $|\lambda|/2\leq |z|\leq 2|\lambda|$, we have $\dist(z,\Lambda) \asymp |z-\lambda|$ and hence
\begin{eqnarray}
I_1\ \asymp\ \int_{|\lambda|/2}^{2|\lambda|} \frac{rdr}{(1+r)^\beta}\ \asymp\ |\lambda|^{2-\beta}.\nonumber
\end{eqnarray}
On the other hand, using the fact that  $|z-\lambda|\asymp|\lambda|$, $\dist(z,\Lambda)\leq |z|$ when $|z|\leq |\lambda|/2$ and that $\dist(z,\Lambda)\leq |z-\lambda|$ for $|z|\geq 2|\lambda|$, we get
\begin{eqnarray}
I_2\ \leq\ \frac{1}{|\lambda|} \int_0^{|\lambda|/2} \frac{r^2dr}{(1+r)^{\beta}}+ \int_{2|\lambda|}^\infty \frac{rdr}{(1+r)^\beta}\ \asymp\ |\lambda|^{2-\beta}.\nonumber
\end{eqnarray}
This ensures the desired result in \eqref{inte}. Now, to prove the estimate \eqref{prb}, we write
{\small\begin{eqnarray}
\sum_{\lambda\in\Lambda}\ |\lambda|^{\nu}\ \frac{\dist(z,\Lambda)}{|z-\lambda|} & = & \left(\sum_{|\lambda|\leq |z|/2}+\sum_{|z|/2\leq |\lambda|\leq 2|z|}+ \sum_{|\lambda| \geq 2|z|}\right)\  |\lambda|^{\nu}\ \frac{\dist(z,\Lambda)}{|z-\lambda|} \nonumber\\
  & = & J_1+J_2+J_3.\nonumber
\end{eqnarray}}
As in the proof of the first estimate, we have
\begin{eqnarray}
J_1+J_3\ \leq\ \sum_{|\lambda|\leq |z|/2}\ |\lambda|^{\nu} + (1+|z|) \sum_{|\lambda|\geq 2|z|} |\lambda|^{\nu-1}\ \asymp\ (1+|z|)^{\nu},\nonumber
\end{eqnarray}
and also
\begin{eqnarray}
J_2\ \asymp\ \sum_{|z|/2\leq|\lambda|\leq 2|z|}\ |\lambda|^{\nu}\  \asymp\ (1+|z|)^{\nu}.\nonumber
\end{eqnarray}
%and finally
%\begin{eqnarray}
%J_3\ \leq\ \ \asymp\ (1+|z|)|z|^{\nu-1}\ \asymp\ (1+|z|)^{\nu}. \nonumber
%\end{eqnarray}
This completes the proof.
\end{proof}

% % % % % % % % % % % % % % % % % % % % % % % % % % % % % % % % % % % % % % % % % %
% % % % % % % % % % % % % % % % % % % % % % % % % % % % % % % % % % % % % % % % % % % % % %ù
% % % % % % % % % % % % % % % % % % % % % % % % % % % % % % % % % % % % % % % % % % % % % % % % % % % % % % % % % % % % % % % % % % % %
% % % % % % % % % % % % % % % % % % % % % % % % % % % % % % % % % % % % % % % % % % % % % % % % % % % %
% % % % % % % % % % % % % % % % % % % % % % % % % % % % % % % % % % % % % % % % % % % % % % %
% % % % % % % % % % % % % % % % % % % % % % % % % % % % % % % % % % % % % % % % % % % % % % % % % % % % % % % % % % % % % % % % % % % % %
% % % % % % % % % % % % % % % % % % % % % % % % % % % % % % % % % % % % % % % % % % % % % % % % % % %\section{Proof of the main results}
\section{Proof of Theorem \ref{thm1}}\label{proof-thm1}
This section is devoted to prove Theorem \ref{thm1}. First, note that a sequence $\Lambda$ of complex numbers is a complete interpolating set for $\cF^p_\varphi$ if and only if the following associated operator
$$\begin{array}{cccl}
    T_\Lambda\ : & \cF^p_\varphi & \longrightarrow & \ell^p   \\
     & f & \longmapsto & \left((1+|\lambda|)^{2/p}\ f(\lambda)\ e^{-\varphi(\lambda)}\right)_{\lambda\in\Lambda}
\end{array}$$
is bounded and invertible.\\

% % % % % % % % % % % % % % % % % % % % % % % % % % % % % % % % % % % % % % % % % % % % % % %
% % % % % % % % % % % % % % % % % % % % % % % % % % % % % % % % % % % % % % % % % % % % % % % % % % % % % % % % % % % % % % % % % % % % % % % % % % % % % % % % % % % % % % % % % % % % % % % % % % % % % % % % %
% % % % % % % % % % % % % % % % % % % % % % % % % % % % % % % % % % % % % % % % % % % % % % % % % % % % % % % % % % % % %
\begin{enumerate}[label=$(\arabic*)$,leftmargin=* ,parsep=0cm,itemsep=0cm,topsep=0cm]
\item Fix $p\in[1,\infty[\setminus\{4/3,4\}$ and let $l\in\{-1,0,1\}$ be the integer index of the corresponding sequence to $p$ as in Theorem \ref{thm1}.\\

First, since the sequence $\Lambda_{l}$ is $d_\rho-$separated, as stated in Lemma \ref{lem4}, there exists $C>0$ such that
\begin{eqnarray}\nonumber
\|f\|^p_{p,\varphi,\Lambda_{l}} := \sum_{\lambda\in\Lambda_{l}}\ (1+|\lambda|^2) \left|f(\lambda)e^{-\varphi(\lambda)}\right|^p\ \leq\ C\ \|f\|^p_{p,\varphi},\quad
\end{eqnarray}
for every $f\in\cF^p_\varphi$ and hence $T_{\Lambda_{l}}$ is bounded. Also, if $f\in\cF^p_\varphi$ is an element of the kernel of $T_{\Lambda_l}$, then $f$ must vanish on $\Lambda$. According to Corollary \ref{coro1} the sequence $\Lambda_{l}$ is a uniqueness set for $\cF^p_\varphi$, therefore $f$ must be identically zero and consequently $T_{\Lambda_{l}}$ is one-to-one from $\cF^p_\varphi$ to $\ell^p$.\\

Now, let us prove that $T_{\Lambda_{l}}$ is onto. For that purpose, let $a=(a_\lambda)_{\lambda\in\Lambda_{l}}\in \ell^p$ and consider the function defined by
\begin{eqnarray}
f_a(z) = \sum_{\lambda\in\Lambda_{l}}\ a_\lambda (1+|\lambda|)^{-2/p}\ e^{\varphi(\lambda)}\ \frac{G_{\Lambda_{l}}(z)}{G'_{\Lambda_{l}}(\lambda)(z-\lambda)},\quad z\in\C,
\end{eqnarray}
where $G_{\Lambda_{l}}$ is the infinite product associated with $\Lambda_{l}$. The series defining $f_a$ converges uniformly on every compact set of $\C$ and hence $f_a$ is an entire function satisfying $T_{\Lambda_{l}} f_a = a$. Indeed, let $r>0$ and $|z|\leq r$, {using \eqref{est2} in Lemma \ref{lem1} we have for every  $|\lambda|\geq 2r$}
\begin{eqnarray}
\left| a_\lambda (1+|\lambda|)^{-2/p}\ e^{\varphi(\lambda)}\ \frac{G_{\Lambda_{l}}(z)}{G'_{\Lambda_{l}}(\lambda)(z-\lambda)}\right|\ \lesssim\ C(r) (1+|\lambda|)^{1/2+l-2/p}.
\end{eqnarray}
This ensures the uniform convergence of the series in compact sets of the complex plane {when l=-1 for every $p$, when $l=0$ for every $p<4$, when $l=1$ for all $p<4/3$}. To complete the proof, it suffices to prove that $f_a\in\cF^p_\varphi$. Indeed, according to the estimates of the function $G_{\Lambda_{l}}$ given in Lemma \ref{lem1}, we obtain
\begin{eqnarray}
\left|f_a(z) e^{-\varphi(z)}\right|^p & \leq &  \left(\sum_{\lambda\in\Lambda_{l}}\ \left|a_\lambda\right| (1+|\lambda|)^{-2/p}\ e^{\varphi(\lambda)}\ \left|\frac{G_{\Lambda_{l}}(z)}{G'_{\Lambda_{l}}(\lambda)(z-\lambda)}\right|\ e^{-\varphi(z)}\right)^p \nonumber \\
   & \asymp & \left(\sum_{\lambda\in\Lambda_{l}}\ \left|a_\lambda\right| |\lambda|^{3/2+l-2/p}\ \frac{\dist(z,\Lambda_{l})}{|z-\lambda|(1+|z|)^{3/2+l}}\right)^p. \label{a}
\end{eqnarray}
From Lemma \ref{integ}, we have
\begin{eqnarray}\label{probab1}
\sum_{\lambda\in\Lambda_{l}}\ |\lambda|^{3/2+l-2/p}\ \frac{\dist(z,\Lambda_{l})}{|z-\lambda|(1+|z|)^{3/2+l-2/p}}\ \asymp\ 1, \quad z\in\C.
\end{eqnarray}
Applying now Jensen's inequality, the identities \eqref{a} and \eqref{probab1} ensure that
\begin{eqnarray}
\left|f_a(z) e^{-\varphi(z)}\right|^p & \lesssim & \left(\sum_{\lambda\in\Lambda_{l}}\ \frac{\left|a_\lambda\right|}{(1+|z|)^{2/p}} |\lambda|^{3/2+l-2/p}\ \frac{\dist(z,\Lambda_{l})}{|z-\lambda|(1+|z|)^{3/2+l-2/p}}\right)^p \nonumber\\
  & \lesssim & \sum_{\lambda\in\Lambda_{l}}\ \frac{\left|a_\lambda\right|^p}{(1+|z|)^{2}} |\lambda|^{3/2+l-2/p}\ \frac{\dist(z,\Lambda_{l})}{|z-\lambda|(1+|z|)^{3/2+l-2/p}}, \quad
\end{eqnarray}
for every $z\in\C.$ Integrating both sides of the last inequality with respect to the Lebesgue measure $dm(z)$, we get
{\begin{eqnarray}
\|f_a\|^p_{p,\varphi} & \lesssim & \sum_{\lambda\in\Lambda_{l}}\ \left|a_\lambda\right|^p |\lambda|^{3/2+l-2/p}\ \int_\C \frac{\dist(z,\Lambda_{l})}{|z-\lambda|(1+|z|)^{2+3/2+l-2/p}}\ dm(z). 
\end{eqnarray}
Setting $\beta=2+3/2+l-2/p$, we deduce from Lemma \ref{integ} that the integral in the last inequality is controlled by a constant times $|\lambda|^{2-\beta}$ when $l=-1$ and $p>4$, when $l=0$ and $p>4/3$ or when $l=1$ and $p\geq 1$. Consequently,
\begin{eqnarray}
\|f_a\|^p_{p,\varphi} & \lesssim & \sum_{\lambda\in\Lambda_{l}}\ \left|a_\lambda\right|^p |\lambda|^{3/2+l-2/p}\ |\lambda|^{2-(2+3/2+l-2/p)}\ \nonumber =\ \|a\|_{p}^p <\infty.
\end{eqnarray}}
%The last equivalence follows from Lemma \ref{integ}, the convergence of the integral is verified in virtue to the relation between $p$ and $l$. 
The proof is complete.\\

% % % % % % % % % % % % % % % % % % % % % % % % % % % % % % % % % % % % % % % % %

% % % % % % % % % % % % % % % % % % % % % % % % % %
% % % % % % % % % % % % % % % % % % % % % % % % % % % % % %
\item Firstly, recall that we assume $\alpha=1$. Now, for $l\in\{0, 1\}$ let $p=4/(2l+1)$. As in the previous case, the $d_\rho-$separation of the sequence $\Sigma=e^{-1/4}\Lambda_{l}$ affirms that $T_\Sigma$ is bounded from $\cF^p_\varphi$ to $\ell^p$. On the other hand, if $G_\Sigma$ is the entire function associated with $\Sigma$, then we have
\begin{eqnarray}
G_\Sigma(z) & = & \prod_{\lambda\in\Lambda_l}\left(1-\frac{z}{e^{-1/4}\lambda}\right) \ =\ G_{\Lambda_l}(e^{1/4}z),\quad z\in\C. \label{estest}
\end{eqnarray}
{Using the estimates of $G_{\Lambda_l}$ given in Lemma \ref{lem1}, we  obtain
\begin{eqnarray}
\left|G_\Sigma(z)\right| & \asymp & e^{\varphi(e^{1/4}z)}\frac{\dist(e^{1/4}z,\Lambda_l)}{(1+|z|)^{3/2+l}} \nonumber\\
 & \asymp & e^{\left(\log|z|+\frac{1}{4}\right)^2}\frac{\dist(z,\Sigma)}{(1+|z|)^{3/2+l}} \nonumber\\
 & \asymp &  e^{\varphi(z)}\  \frac{\dist(z,\Sigma)}{(1+|z|)^{l+1}},
\end{eqnarray}
for every $z\in\C.$ Analogously, we have \begin{eqnarray}
\left|G'_\Sigma(\sigma)\right|\ \asymp\ \frac{e^{\varphi(\sigma)}}{(1+|\sigma|)^{l+1}},\quad \sigma\in\Sigma. \label{sigma}
\end{eqnarray}} Following the same steps of the proof of Corollary \ref{coro1}, one can easily show that $\Sigma$ is a uniqueness set for $\cF^p_\varphi$. This implies that $T_\Sigma$ is a one-to-one operator from $\cF^p_\varphi$ to $\ell^p$. Now, and in order to prove that $T_\Sigma$ is onto, take $a=\left(a_\sigma\right)_{\sigma\in\Sigma}$ in $\ell^p$ and consider the function
\begin{eqnarray}\label{ser}
f_a(z)\ =\ \sum_{\sigma\in\Sigma}\ a_\sigma\  (1+|\sigma|)^{-2/p}e^{\varphi(\sigma)}\ \frac{G_\Sigma(z)}{G'_\Sigma(\sigma)(z-\sigma)}, \quad z\in\C.
\end{eqnarray}
Now, fix $r>0$.  According to the estimate in \eqref{sigma}, we have for every  $|z|\leq r$ and   every $|\sigma|\geq 2r$
\begin{eqnarray}
\left| a_\sigma {(1+|\sigma|)^{-2/p}}\ e^{\varphi(\sigma)}\ \frac{G_{\Sigma_{l}}(z)}{G'_{\Sigma_{l}}(\sigma)(z-\sigma)}\right|\ \lesssim\ C(r) {(1+|\sigma|)}^{l-2/p}.
\end{eqnarray}
Since $p=4/(2l+1)$ we obtain $l-2/p=-1/2$. This ensures the convergence of the series in \eqref{ser} uniformly in compact sets of $\C$. Using the same arguments of the previous proof one can prove that the function $f_a$ belongs to $\cF^p_\varphi$ and satisfies $(1+|\sigma|)^{2/p}f_a(\sigma)\ e^{-\varphi(\sigma)}\ =\ a_\sigma$, for every $\sigma\in\Sigma$, i.e. $T_\Sigma f_a=a.$ This completes the proof. 
\end{enumerate}
% % % % % % % % % % % % % % % % % % % % % % % % % % % % % % % % % %
% % % % % % % % % % % % % % % % % % % % % % % % % % % % % % % % % % % % % % % % %
\section{Proof of Theorems \ref{thm3} and \ref{thm2}}\label{proof-thm-3}
% % % % % % % % % % % % % % % % % % % % % % % % % % % % % % % % % % % % %
% % % % % % % % % % % % % % % % % % % % % % % % % % % % % % % % % % % % %

In this section we prove Theorems \ref{thm3} and \ref{thm2}. First, we prove that our conditions are sufficient in both situations and next that they are necessary. The proofs of our results appear generally like that of \cite[Theorem 1.1]{baranov2015sampling}, we will focus on the parts when the two proofs differ. In what follows we use the notation $\Gamma_l=\{\gamma_n\ :\ n\geq l\}$, where $l$ is an integer.

Now, Let $\Gamma=\{\gamma_n\ :\ n\geq 0\}$ be a countable set of $\C$. Without loss of generality, we can suppose that $\Gamma\cup\{*\}=\Gamma\cup\{\gamma_{-1}\}=:\Gamma_{-1}$, {where $\gamma_{-1}$ is any point from $\C$ such that $|\gamma_{-1}|<|\gamma_{0}|$}. Also $\Gamma=\Gamma_0$ and $\Gamma\setminus\{*\}=\Gamma\setminus\{\gamma_{0}\}=\Gamma_1$.

\subsection{Sufficient Conditions}
%\begin{proof}

\begin{enumerate}[label=$(\arabic*)$,leftmargin=* ,parsep=0cm,itemsep=0cm,topsep=0cm]
\item\label{1/} As a first point we state sufficient conditions of Theorem \ref{thm3}. To this end,  fix $p\in[1,\infty[\setminus\{4/3,4\}$ and let $l\in\{-1,0,1\}$ be the integer index of the corresponding sequence as mentioned above.\\

Now, since the sequence $\Gamma_{l}$ is $d_\rho-$separated, $T_{\Gamma_{l}}$ maps $\cF^p_\varphi$ boundedly on $\ell^p$. On the other hand, take $f\in\cF^p_\varphi$ such that $T_{\Gamma_l}f = 0$, then $f$ must vanish on $\Gamma_l$. Using Hadamard's factorization theorem we can write $f=hG_{\Gamma_l}$, for some entire function $h$. According to Lemmas \ref{estimate} and \ref{lem0} the following estimate holds
\begin{eqnarray}\label{aaa}
\frac{|h(z)|\dist(z,\Gamma_l)}{1+|z|^{3/2+2\delta+l}}
 & \lesssim&  |h(z)G_{\Gamma_l(z)}|e^{-\varphi(z)}\nonumber\\
 & = &\ |f(z)|e^{-\varphi(z)}\nonumber\\
 & \lesssim& (1+|z|)^{-2/p},\quad z\in\C.
\end{eqnarray} Since $\dist(z,\Gamma_l)\ \asymp\ 1+|z|$ in  {$\C\setminus\underset{\gamma\in\Gamma_l}{\cup}\D(\gamma,\beta|\gamma|)$, for some $\beta\in(0,1)$}, $h$ is a polynomial. As $f\in\cF^p_\varphi$, using the two first lines in \eqref{aaa} we get
\[\int_{\C}\ \left(\frac{|h(z)|\dist(z,\Gamma_l)}{(1+|z|)^{3/2+2\delta+l}}\right)^p\ dm(z)\ \leq \int_{\C}|f(z)|^pe^{-p\varphi(z)}dm(z)\ <\ \infty.\]
{Taking into account Lemma \ref{separation} we get $\nu=2\delta+l>2/p-1/2$, i.e. $\delta>1/p-1/4-l/2$. Now, the conditions required in Theorem \ref{thm3} in the different ranges of $p$ depending on $l$ allow us to conclude} that $h$ is the zero function and hence $f$ is identically zero too. %, Lemmas \ref{estimate} and  \ref{separation}, and simple calculations ensure that the sequence $\Gamma_{l}$ is a uniqueness set of zero excess for $\cF^p_\varphi$ and 
That is, $T_{\Gamma_{l}}$ is a bounded one-to-one operator from $\cF^p_\varphi$ to $\ell^p$. To prove that $T_{\Gamma_{l}}$ is onto, let $a=\left(a_n\right)_{n\geq l}$ in $\ell^p$ and consider the function
\begin{eqnarray}
  g_a(z)\ =\ \sum_{n\geq l}\ a_n (1+|\gamma_n|)^{-2/p} e^{\varphi(\gamma_n)}\  \frac{G_{\Gamma_{l}}(z)}{G'_{\Gamma_{l}}(\gamma_n)(z-\gamma_n)},\quad z\in\C.
\end{eqnarray}
The last series converges uniformly on every compact set of the complex plane. Indeed, let $r>0$ and $|z|\leq r$. {By \eqref{est2} of Lemma \ref{lem1} we have for every} $|\gamma_n|\geq 2r$
\begin{eqnarray}
\left|\ a_n (1+|\gamma_n|)^{-2/p}\ e^{\varphi(\gamma_n)}\ \ \frac{G_{\Gamma_{l}}(z)}{G'_{\Gamma_{l}}(\gamma_n)(z-\gamma_n)}\ \right|\ \lesssim\ C(r)\ |\gamma_n|^{1/2+l+2\delta-2/p}.
\end{eqnarray}
Again, the condition on $\delta = \Delta_N$ (taking account of the relation between $p$ and the corresponding $l$) ensures that $g_a$ is an entire function satisfying $\left(1+\left|\gamma_n\right|\right)^{2/p}g_a(\gamma_n)\ e^{-\varphi(\gamma_n)}\ =\ a_n$, for every $n\geq l$. It remains to check that $g_a\in\cF^p_\varphi$. To this end, since $\Lambda_{l}$ is a complete interpolating sequence for $\cF^p_\varphi$ (see Theorem \ref{thm1}), we immediately obtain
\begin{eqnarray}
\|g_a\|^p_{p,\varphi} & \asymp & \sum_{m\geq l}\ |\lambda_m|^2\ \left|g_a(\lambda_m)\ e^{-\varphi(\lambda_m)}\right|^p \nonumber \\
   & = & \sum_{m\geq l}\ |\lambda_m|^2 \left|\sum_{n\geq l}\ a_n (1+|\gamma_n|)^{-2/p} e^{\varphi(\gamma_n)}\ \frac{G_{\Gamma_{l}}(\lambda_m)}{G'_{\Gamma_{l}}(\gamma_n)(\lambda_m-\gamma_n)} e^{-\varphi(\lambda_m)}\right|^p \nonumber \\
   & = & \sum_{m\geq l} \ \left|\sum_{n\geq l}\ a_n \left(\frac{|\lambda_m|}{1+|\gamma_n|}\right)^{2/p} e^{\varphi(\gamma_n)-\varphi(\lambda_m)}\  \frac{G_{\Gamma_{l}}(\lambda_m)}{G'_{\Gamma_{l}}(\gamma_n)(\lambda_m-\gamma_n)} \right|^p. \nonumber
\end{eqnarray}
It follows from this identity that the functions $(g_a)_{a\in\ell^p}$ belong to $\cF^p_\varphi$ if and only if the matrix $A_p=\left(A_{p,n,m}\right)_{n,m}$ maps $\ell^p$ continuously onto itself,  where
\begin{eqnarray}
A_{p,n,m}\ :=\ \left(\frac{|\lambda_m|}{1+|\gamma_n|}\right)^{2/p} e^{\varphi(\gamma_n)-\varphi(\lambda_m)}\  \frac{G_{\Gamma_{l}}(\lambda_m)}{G'_{\Gamma_{l}}(\gamma_n)(\lambda_m-\gamma_n)},
\end{eqnarray}
for every $ n,m\geq l.$ On the other hand, note that there exists a polynomial $P_l$ if $l=0,1$ (or a fraction if $l=-1$) of degree $|l|$ satisfying $G_{\Gamma_0}(z)=G_{\Gamma_l}(z)P_l(z)$. Hence we have $G'_{\Gamma_0}(\gamma_n)=G'_{\Gamma_l}(\gamma_n)P_l(\gamma_n)$, for every $n\geq 0$. Using this simple fact, we obtain
\begin{eqnarray}
\left|A_{p,n,m}\right| & \asymp & \left|\frac{\lambda_m}{\gamma_n}\right|^{2/p} e^{\varphi(\gamma_n)-\varphi(\lambda_m)}\  \left|\frac{G_{\Gamma_{0}}(\lambda_m)}{G'_{\Gamma_{0}}(\gamma_n)(\lambda_m-\gamma_n)}\frac{P_l(\gamma_n)}{P_l(\lambda_m)}\right| \nonumber\\
   & \asymp & \left|\frac{\lambda_m}{\gamma_n}\right|^{2/p-l} e^{\varphi(\gamma_n)-\varphi(\lambda_m)}\  \left|\frac{G_{\Gamma_{0}}(\lambda_m)}{G'_{\Gamma_{0}}(\gamma_n)(\lambda_m-\gamma_n)}\right|. \label{matrix_estimate}
  % & = & \left|\frac{\lambda_m}{\gamma_n}\right|^{2/p-l-1} \left|A_{2,n,m}\right| ,\quad n,m\geq l. 
\end{eqnarray}
According to the two first lines in \eqref{A} we have 
\begin{eqnarray}
\left|G_{\Gamma_0}(\lambda_m)\right|\ \asymp\ \frac{\dist(\lambda_m,\Gamma_0)}{|\lambda_m|}\ \prod_{0\leq k\leq m-1}\left|\frac{\lambda_m}{\gamma_k}\right|,\label{prod1}
\end{eqnarray}
and
\begin{eqnarray}
\left|G'_{\Gamma_0}(\gamma_n)\right|\ \asymp\ \frac{1}{|\gamma_n|}\ \prod_{0\leq k\leq n-1}\left|\frac{\gamma_n}{\gamma_k}\right|.\label{prod2}
\end{eqnarray}
Consequently,
\begin{eqnarray}
\left|A_{p,n,m}\right| & \asymp & \left|\frac{\lambda_m}{\gamma_n}\right|^{2/p-l-1} e^{\varphi(\gamma_n)-\varphi(\lambda_m)}\ \frac{\dist(\lambda_m,\Gamma_0)}{|\lambda_m-\gamma_n|}\ \left(\prod_{0\leq k\leq m-1}\left|\frac{\lambda_m}{\gamma_k}\right|\right)\left(\prod_{0\leq k\leq n-1}\left|\frac{\gamma_n}{\gamma_k}\right|\right)^{-1}\nonumber\\
& = &  \frac{\dist(\lambda_m,\Gamma_0)}{|\lambda_m-\gamma_n|}\ e^{\alpha(n,m)},
\end{eqnarray}
where 
\begin{align*}
\alpha(n,m) & = \frac{m(m+1)}{2}-n\left(\frac{n+1}{2}+\delta_n\right) + \left(\frac{n+1}{2}+\delta_n\right)^2- \left(\frac{m+1}{2} \right)^2 \\
 & + \left(\frac{1}{p}-\frac{l+1}{2}\right)(m-n) + \sum_{k=l}^{n-1}\left(\frac{k+1}{2}+\delta_k\right) - \sum_{k=l}^{ m-1}\left(\frac{k+1}{2}+\delta_k\right)\\
 & = \left(\frac{1}{p}-\frac{1}{4}-\frac{l+1}{2}\right)(m-n) + \sum_{k=l}^{n-1} \delta_k - \sum_{k=l}^{ m-1}\delta_k.
\end{align*}
Since the sequence $\left(\delta_n\right)$ is bounded, there exists an integer $M$ such that $\left|\gamma_n-\lambda_m\right|\asymp \lambda_m$ if $m> n+M$, $\left|\gamma_n-\lambda_m\right|\asymp |\gamma_n|$ if $n> m+M$, and $|\gamma_n|\asymp\lambda_m$ for every $|n-m|\leq M$. Also, since our spaces are rotation invariant we can assume that $\dist(\lambda_m,\Gamma_0)\asymp|\lambda_m|$.  Otherwise, we may replace if necessary $\lambda_m$ by $\lambda_me^{i\theta_m}$, for an appropriate $\theta_m\in\R$.  Consequently,\\
$\star$ if $|n-m|\leq M$ we get $\left|A_{p,n,m}\right|\lesssim 1$.\\
$\star$ if $m> n+M$ we then obtain
\begin{eqnarray*}
\left|A_{p,n,m}\right| & \asymp & \exp\left[-\left(\frac{1}{4}-\frac{1}{p}+\frac{l+1}{2}\right)|m-n|-\sum_{k=n}^{m-1}\delta_k\right] \\
 & \lesssim & \exp\left[-\left(\frac{1}{4}-\frac{1}{p}+\frac{l+1}{2}+\nabla_N\right)|m-n|\right].
\end{eqnarray*}
$\star$ if $n>m+M$ we obtain
\begin{eqnarray*}
\left|A_{p,n,m}\right| & \asymp & \exp\left[-\left(\frac{1}{4}+\frac{1}{p}-\frac{l+1}{2}\right)|m-n|+\sum_{k=m}^{n-1}\delta_k\right] \\
 & \lesssim & \exp\left[-\left(\frac{1}{4}+\frac{1}{p}-\frac{l+1}{2}-\Delta_N\right)|m-n|\right].
\end{eqnarray*}
Thus,\\
%According to the proof given in \cite[Page. 17]{baranov2015sampling} for the case $p=2$, it was shown that
%\begin{eqnarray}
%|A_{2,n,m}| & \asymp & \exp\left(-\frac{|m-n|}{4}\pm\sum_{k=m+1}^{n}\delta_k\right) \nonumber\\
%          & \lesssim& e^{-(1/4-\delta)|n-m|},\quad n,m\geq 0, 
%\end{eqnarray}
%where $\delta=\Delta_N$ for some integer $N\geq 1$, the sum in the last identity takes the positive sign if $m\leq n$ and the negative sign if $m\geq n$. Consequently \eqref{matrix_estimate} implies
%\begin{eqnarray}\label{matr-A-estimate}
%\left|A_{p,n,m}\right| & \lesssim & e^{(1/p-(l+1)/2)(m-n)-(1/4-\delta)|n-m|}.
%\end{eqnarray}
$\bullet$ If $l=-1$ and $p>4$, then \eqref{(C)} ensures that $\nabla_N>1/p-1/4$ and $\Delta_N<1/4+1/p$,  for some $N\in\mathbb{N}$, so that $c:=1/4-1/p+\nabla_N>0$ and $c':=1/4+1/p-\Delta_N>0$ and hence 
$$\left|A_{p,n,m}\right|  \lesssim \left\{
  \begin{array}{ll}
     e^{-(1/4-1/p+\nabla_N)|n-m|}\ =\  e^{-c|n-m|}, & \hbox{$m\geq n$;} \\[0.2cm]
     e^{-(1/4+1/p-\Delta_N)|n-m|}\ =\  e^{-c'|n-m|}, & \hbox{$n\geq m$.}
  \end{array}
\right.
$$
$\bullet$ If $l=0$ and $4/3<p<4$, then \eqref{(B)} implies that $1/4-1/q<\nabla_N\leq \Delta_N<1/p-1/4$, so that $c:=1/q-1/4+\nabla_N>0$ and $c':=1/p-1/4-\Delta_N>0$, and consequently
$$\left|A_{p,n,m}\right|  \lesssim \left\{
  \begin{array}{ll}
     e^{-(1/q-1/4+\nabla_N)|n-m|}\ =\  e^{-c|n-m|}, & \hbox{$m\geq n$;} \\[0.2cm]
     e^{-(1/p-1/4-\Delta_N)|n-m|}\ =\  e^{-c'|n-m|}, & \hbox{$n\geq m$.}
  \end{array}
\right.
$$
$\bullet$ The remaining case is $l=1$ and $1\leq p<4/3$, the condition required in \eqref{(A)} implies that $c:=1/q+1/4+\nabla_N>0$ and $c':=1/4-1/q-\Delta_N>0$. Thus, 
$$\left|A_{p,n,m}\right|  \lesssim \left\{
  \begin{array}{ll}
     e^{-(1/q+1/4+\nabla_N)|n-m|}\ =\  e^{-c|n-m|}, & \hbox{$m\geq n$;} \\[0.2cm]
     e^{-(1/4-1/q-\Delta_N)|n-m|}\ =\  e^{-c'|n-m|}, & \hbox{$n\geq m$.}
  \end{array}
\right.
$$
This shows that the matrix $A_p$ is bounded in the different cases, which ends the proof of
the sufficient condition.\\

\item\label{2/} Let $l\in\{0, 1\}$,  $p=4/(2l+1)$. Denote by $\widehat{\Gamma}_l = e^{-1/4}\Gamma_{l} = \{\nu_n\ :\ n\geq l\}$ and $\Sigma\ =\ e^{-1/4}\Lambda_{l}\ =\  \{\sigma_n\ :\ n\geq l\}$. First since $\widehat{\Gamma}_l$ is $d_\rho-$separated, $T_{\widehat{\Gamma}_l}$ is bounded from $\cF^p_\varphi$ to $\ell^p$ (see Lemma \ref{lem4}). On the other hand, remark that $G_{\widehat{\Gamma}_l}(z)=G_{\Gamma_l}(e^{1/4}z)$, for every $z\in\C$. According to this remark and to Lemma \ref{estimate}, $G_{\widehat{\Gamma}_l}$ satisfies the estimates
\begin{eqnarray}\nonumber
\frac{\dist(z,\widehat{\Gamma}_l)}{1+|z|^{1+l+2\delta}}\ \lesssim\ \left|G_{\widehat{\Gamma}_l}(z)\right|e^{-\varphi(z)}\ \lesssim\ \frac{\dist(z,\widehat{\Gamma}_l)}{1+|z|^{1+l+2\eta}},\qquad z\in\C,
\end{eqnarray}
where $\delta=\Delta_N$ and $\eta=\nabla_N$. Using the above estimates and following the same steps of the proof in \ref{1/}, one can conclude easily that $T_{\widehat{\Gamma}_l}$ is one-to-one. Now, to check that $T_{\widehat{\Gamma}_l}$ is onto we associate with each $a=(a_n)_{n\geq l}\in\ell^p$ the following function
\begin{eqnarray}
g_a(z)\ =\ \sum_{n\geq l}\  a_n{(1+|\nu_n|)}^{-2/p}e^{\varphi(\nu_n)}\ \frac{G_{\widehat{\Gamma}_l}(z)}{G'_{\widehat{\Gamma}_l}(\nu_n)(z-\nu_n)},\quad z\in\C.\label{seer}
\end{eqnarray}
Again, since $G_{\widehat{\Gamma}_l}(z)=G_{\Gamma_l}(e^{1/4}z)$, for every $z\in\C$, Lemma \ref{estimate} implies that 
\begin{eqnarray}\nonumber
\frac{e^{\varphi(\nu_n)}}{1+|\nu_n|^{1+l+2\delta}}\ \lesssim\ \left|G'_{\widehat{\Gamma}_l}(\nu_n)\right|\  \lesssim\ \frac{e^{\varphi(\nu_n)}}{1+|\nu_n|^{1+l+2\eta}},\quad \nu_n\in\widehat{\Gamma}_l.
\end{eqnarray}
Thus, for every $r>0$ and every $|z|\leq r$ we  have
\begin{eqnarray}
\left|a_n(1+|\nu_n|)^{-2/p}e^{\varphi(\nu_n)}\ \frac{G_{\widehat{\Gamma}_l}(z)}{G'_{\widehat{\Gamma}_l}(\nu_n)(z-\nu_n)}\right|\ \lesssim\ C(r) (1+|\nu_n|)^{-2/p+l+2\delta},\quad |\nu_n|\geq 2r.\nonumber
\end{eqnarray}
Since $p=4/(2l+1)$ we obtain $\beta:=-2/p+l+2\delta=2\delta-1/2$. The condition on $\delta=\Delta_N$ implies  that the above series converges uniformly on compact sets of $\C$ and hence $g_a$ is  holomorphic in $\C$ and satisfies $T_{\widehat{\Gamma}_l}\ g_a\ =\ a.$ To prove that $g_a\in\cF^p_\varphi$ we use the fact that $\Sigma$ is a complete interpolating sequence for $\cF^p_\varphi$ (see Theorem \ref{thm1}), and we obtain
\begin{eqnarray}
\|g_a\|^p_{p,\varphi}
   & \asymp & \sum_{m\geq l}\ |\sigma_m|^2|g_a(\sigma_m)|^p \ e^{-p\varphi(\sigma_m)} \nonumber\\
   & = & \sum_{m\geq l}\  \left|\sum_{n\geq l}\ a_n \left(\frac{|\sigma_m|}{{1+|\nu_n|}}\right)^{2/p} e^{\varphi(\nu_n)-\varphi(\sigma_m)}\ \frac{G_{\widehat{\Gamma}_l}(\sigma_m)}{G'_{\widehat{\Gamma}_l}(\nu_n)(\sigma_m-\nu_n)} \right|^p. \nonumber
\end{eqnarray}
Once more, the functions $\left(g_a\right)_{a\in\ell^p}$ belong to $\cF^p_\varphi$ if and only if the matrix $A_p=\left(A_{p,n,m}\right)_{n,m\geq l}$ acts continuously on $\ell^p$, where % On the other hand, we have
\begin{eqnarray}
|A_{p,n,m}| & = &  \left|\frac{\sigma_m}{\nu_n}\right|^{2/p} e^{\varphi(\nu_n)-\varphi(\sigma_m)} \left|\frac{G_{\widehat{\Gamma}_l}(\sigma_m)}{G'_{\widehat{\Gamma}_l}(\nu_n)(\sigma_m-\nu_n)}\right|. \nonumber
\end{eqnarray}
Again, note that there exists a polynomial $P_l$ of degree $l\in\{0,1\}$ satisfying $G_{\Gamma_0}(e^{1/4}z)=G_{\widehat{\Gamma}_l}(z)P_l(z)$ and hence we have $G'_{\Gamma_0}(\gamma)=G'_{\widehat{\Gamma}_l}(\nu)P_l(\nu)$ for every $\gamma=e^{1/4}\nu\in \Gamma_{l} $. Using this simple remark and observing that $|\lambda_m|\asymp|\sigma_m|$ and $|\nu_n|\asymp|\gamma_n|$, we get
\begin{eqnarray}
|A_{p,n,m}| & \asymp &  \left|\frac{\lambda_m}{\gamma_n}\right|^{2/p} e^{\varphi(e^{-1/4}\gamma_n)-\varphi(e^{-1/4}\lambda_m)} \left|\frac{P_l(\gamma_n)}{P_l(\lambda_m)}\right|\left|\frac{G_{\Gamma_0}(\lambda_m)}{G'_{\Gamma}(\gamma_n)(\lambda_m-\gamma_n)}\right| \nonumber\\
  & \asymp & \left|\frac{\lambda_m}{\gamma_n}\right|^{2/p-l} e^{\left(\log|\gamma_n|-1/4\right)^2-\left(\log|\lambda_m|-1/4\right)^2} \left|\frac{G_{\Gamma_0}(\lambda_m)}{G'_{\Gamma}(\gamma_n)(\lambda_m-\gamma_n)}\right| \nonumber\\
  & = & \left|\frac{\lambda_m}{\gamma_n}\right|^{2/p-l+1/2} e^{\varphi(\gamma_n)-\varphi(\lambda_m)} \left|\frac{G_{\Gamma_0}(\lambda_m)}{G'_{\Gamma}(\gamma_n)(\lambda_m-\gamma_n)}\right| \nonumber\\
  & = & \left|\frac{\lambda_m}{\gamma_n}\right|^{2/p-l+1/2-1} |A_{2,n,m}|\ \asymp\ |A_{2,n,m}|, \label{matrix_estimate'}
\end{eqnarray}
since $p=4/(2l+1)$. From the calculations of the previous proof we have
\begin{eqnarray}
|A_{2,n,m}| & \asymp & \exp\left(-\frac{|m-n|}{4}\pm\sum_{k=m+1}^{n}\delta_k\right), \nonumber
\end{eqnarray}
{where the sum in the last identity takes the positive sign if $m\leq n$ and the negative sign if $m\geq n$.} This implies the desired result.
\end{enumerate}

% % % % % % % % % % % % % % % % % % % % % % % % % % % % % % % % % % % % % % % % % % % % % % % % %
% % % % % % % % % % % % % % % % % % % % % % % % % % % % % % % % % % % % % % % % % % % % % % % %
% % % % % % % % % % % % % % % % % % % % % % % % % % % % % % % % % % % % % % % % % % % % % % % % % % % % % % % % % %
\subsection{Necessary conditions}
Suppose that $\Gamma$ is a complete interpolating sequence for $\cF^p_\varphi$.\\

\begin{enumerate}[label=$(\alph*)$,leftmargin=* ,parsep=0cm,itemsep=0cm,topsep=0cm]
    \item Since $\Gamma$ is an interpolating sequence, classical arguments show that $\Gamma$ is $d_\rho-$separated, see \cite[Lemma 2.1 and Corollary 2.3]{baranov2015sampling}.\\ %(see \cite[Lemma 4.1]{BORICHEV2007563}).\\

   \item Let $\Lambda=\{\lambda_n=e^{\frac{n+1}{2}}\ :\ n\geq l\}$ and let $\Gamma=\{\gamma_n\ :\ n\geq l\}$ be a complete interpolating sequence for $\cF^p_\varphi$ such that $|\gamma_n|\leq |\gamma_{n+1}|$. We write $\left|\gamma_n\right|=\lambda_ne^{\delta_n}$, {and we will prove} that $\left(\delta_n\right)$ is a bounded sequence. \\

First, for every $\gamma\in\Gamma$ there exists a unique function $f_\gamma\in\cF^p_\varphi$ {which is a solution} of the interpolating problem $f_\gamma(\gamma)\ =\ 1$ and {$f_{\gamma}{\left|\right.}_{\Gamma\setminus\{\gamma\}}\ =\ 0$}. The zero set of $f_\gamma$ is exactly $\Gamma\setminus\{\gamma\}.$ Indeed, if not, this contradicts the fact that $\Gamma$ is a uniqueness set for $\cF^p_\varphi$. The Hadamard factorization theorem ensures that $f_\gamma(z)\ =\ c\ \frac{G_\Gamma(z)}{G'_\Gamma(\gamma)(z-\gamma)}$ for some constant $c\in\C$, and $G_\Gamma$ is the infinite product associated with $\Gamma.$ Now, since $f_\gamma(\gamma)\ =\ 1$, we directly get  $f_\gamma\ =\ \frac{G_\Gamma}{G'_\Gamma(\gamma)(.\ -\ \gamma)}\in\cF^p_\varphi$. Consequently, Lemma \ref{lem0} implies that
\begin{eqnarray}\label{1}
|f_\gamma(z)|\ \lesssim\  \|f_\gamma\|_{p,\varphi}\ \frac{e^{\varphi(z)}}{1+|z|^{2/p}},\quad z\in\C.
\end{eqnarray}
On the other hand, the sequence $\Gamma$ is a complete interpolating set for $\cF^p_\varphi$. Hence
\begin{eqnarray}\label{2}
\|f_\gamma\|^p_{p,\varphi}\ &\asymp& \sum_{\gamma'\in\Gamma}\ |\gamma'|^{2}\ \left|f_\gamma(\gamma')\ e^{-\varphi(\gamma')}\right|^p\nonumber \\
&=&\ |\gamma|^{2}\  e^{-p\varphi(\gamma)},\quad \gamma\in\Gamma.
\end{eqnarray}

Now, assume that $\left(\delta_n\right)$ contains a sub-sequence $(\delta_{n_k})$ which tends to infinity. Thus, for every $k$ there exists $m_k$ such that $d_\rho(\gamma_{n_k},\lambda_{m_k}) \lesssim 1$ and $|n_k-m_k|\rightarrow\infty$.  Since the sequence $\Gamma$ is $d_\rho-$separated, {a similar reasoning as for \eqref{prod1} and \eqref{prod2} gives}
\begin{align}
|G_\Gamma(\lambda_{m_k})|\ &\asymp\ \frac{|\lambda_{m_k}-\gamma_{n_k}|}{\lambda_{m_k}}\ \prod_{j=l}^{n_k-1}\left|\frac{\lambda_{m_k}}{\gamma_j}\right|,\label{3} \\
 |G'_\Gamma(\gamma_{m_k})|\ &\asymp\ \frac{1}{|\gamma_{n_k}|} \prod_{j=l}^{n_k-1}\left|\frac{\gamma_{n_k}}{\gamma_j}\right|. \label{44}
\end{align}
Combining \eqref{1}, \eqref{2}, \eqref{3}, and \eqref{44} we obtain
\begin{eqnarray}\label{4}
\left|\frac{\lambda_{m_k}}{\gamma_{n_k}}\right|^{n_k}\ \asymp\ \left| f_{\gamma_{n_k}}(\lambda_{m_k}) \right|\ \lesssim \ e^{\varphi(\lambda_{m_k})-\varphi(\gamma_{n_k})}.
\end{eqnarray}
Without loss of generality we may suppose that $\delta_{n_k}\rightarrow \infty$, we assume also that $|\gamma_{n_k}|\geq e^{2}\lambda_{m_k}$. Otherwise, we replace {$m_k$ by $m_k-m_k'$ for an adequate integer $m_k'$} (the case $\delta_{n_k}\rightarrow-\infty$ is similar). {By our assumption $d_\rho(\gamma_{n_k},\lambda_{m_k})\lesssim1$ and $|\gamma_{n_k}|=|\lambda_{n_k}|e^{\delta_{n_k}}$, there exists a sequence $(\eta_{n_k,m_k})\subset [1,2]$ such that $\frac{1+n_k}{2}+\delta_{n_k}=\frac{1+m_k}{2}+\eta_{n_k,m_k}$, i.e. $n_k+2\delta_{n_k}=m_k+2\eta_{n_k,m_k}$.} By simple calculations we get
\begin{align*}
A & :=\ n_{k}\left( \frac{m_k-n_k}{2}-\delta_{n_k}\right) - \left(\frac{m_k+1}{2}\right)^2 + \left(\frac{n_k+1}{2}+\delta_{n_k} \right)^2  \\
  & =\ \delta_{n_k}^2 + \delta_{n_k} + \frac{(n_k-m_k)}{2}\frac{(m_k-n_k+2)}{2}  \\
     & =\ \delta_{n_k}^2 + \delta_{n_k} + (\eta_{n_k,m_k}-\delta_{n_k})(\delta_{n_k}-\eta_{n_k,m_k}+1)  \\
   & =\ 2\delta_{n_k}\eta_{n_k,m_k} + O(1).
\end{align*}
%where $\eta_{n_k,m_k}$ is a positive real sequence such that $\eta_{n_k,m_k}\asymp 1$. 
{Passing to logarithms in \eqref{4} we reach a contradiction to the above inequality.}\\

Simple modifications in this proof ensure the result for the cases $p=4/3\ \mbox{and}\ 4$ in Theorem \ref{thm2}.\\

\item For $p\in\{4/3,4\}$, assume that for every $N\geq 1$ we have
     \begin{eqnarray}
     \sup_{n}\ \frac{1}{N}\left|\sum_{k=n+1}^{n+N}\ \delta_k \right|\ =\ \frac{1}{4} + \varepsilon_N, \nonumber
     \end{eqnarray}
for some nonnegative sequence $\left(\varepsilon_N\right)$. By the supremum property, for every $N\geq 1$ there exists an integer $n_N$ such that
\begin{eqnarray}
\left|\sum_{k=n_N+1}^{n_N+N}\ \delta_k \right|\ \geq\ N\left(\frac{1}{4} + \varepsilon_N\right) - 1. \nonumber
\end{eqnarray}
Also by definition, for every $K\leq N$
\begin{eqnarray}
\left|\sum_{k=n_N+1}^{n_N+K}\ \delta_k \right|\ \leq\ K\left(\frac{1}{4} + \varepsilon_K\right). \nonumber
\end{eqnarray}
It follows that 
\begin{eqnarray}
\left|\sum_{k=n_N+K+1}^{n_N+N}\ \delta_k \right|\ \geq\ \frac{N-K}{4} + N\varepsilon_N-K\varepsilon_K-1. \nonumber
\end{eqnarray}
 %then according to the result of \cite[Theorem 1.1]{baranov2015sampling}, the matrix $A_2=\left(A_{2,n,m}\right)_{n,m}$ does not define a bounded operator in $\ell^2$.
%According to the first subsection, 
Since $(\delta_n)$ is bounded, we have proved in \eqref{matrix_estimate'} that
%\[\left|A_{p,n,m}\right|\ \asymp\ \left|A_{2,n,m}\right|,\quad n,m\geq l.\]
%Also, we have
\begin{eqnarray}
\left|A_{p,n,m}\right|\  \asymp\ \left|A_{2,n,m}\right|\ \asymp\ \exp\left(-\frac{|m-n|}{4}\pm\sum_{k=m+1}^{n}\delta_k\right),\quad n,m\geq 0, \label{theA2Matrix}
\end{eqnarray}
where the sum in the last identity takes the positive sign if $m\leq n$ and the negative sign if $m\geq n$. \\

Now, if the sequence $\left(N\varepsilon_N\right)$ contains a subsequence which tends to infinity, then for $N$ in this subsequence and fixed $K$ we get that the sum $\left|A_{p,n_N+K,n_N+N}\right|+\left|A_{p,n_N+N,n_N+K}\right|$ is unbounded. Hence the matrix $A_p=\left(A_{p,n,m}\right)$ cannot define a bounded operator in $\ell^p$.

Suppose now that the sequence $\left(N\varepsilon_N\right)$ is bounded. Let $N\geq 1$, if $\underset{k=n_N+1}{\overset{n_N+N}{\sum}}\ \delta_k >0$ then for every $0<K<N$ we have
$$\sum_{k=n_N+K+1}^{n_N+N}\ \delta_k\ \geq\ \frac{N-K}{4} + N\varepsilon_N-K\varepsilon_K-1.$$
Hence, $|A_{p,n_N+N,n_N+K+1}|\ \asymp\ 1$. Also, if $\underset{k=n_N+1}{\overset{n_N+N}{\sum}}\  \delta_k <0$ then for every $0<K<N$ we have
$$\sum_{k=n_N+1}^{n_N+N-K}\ \delta_k\ \leq\ -\frac{N-K}{4} + O(1).$$
Consequently, $|A_{p,n_N+1,n_N+N-K}|\ \asymp\ 1$. %In both cases, the matrix $A_p=\left(A_{p,n,m}\right)_{n,m}$ cannot define a bounded operator in $\ell^p$.
{In both cases, the matrix $A_p=\left(A_{p,n,m}\right)_{n,m}$ contains an  increasing number of entries in one line which are bounded away from zero, and hence it cannot define a bounded operator in $\ell^p$.} This completes the proof of Theorem \ref{thm2}.\\ 

Let us turn to Theorem \ref{thm3} and suppose first that $4/3<p<4$. Assume also that 
\begin{eqnarray*}
\Delta_N=\frac{1}{p}-\frac{1}{4}+\varepsilon_N\quad \mbox{or}\quad \nabla_N=\frac{1}{4}-\frac{1}{q}-\eta_N,
\end{eqnarray*}
where either $(\varepsilon_N)$ or $\left(\eta_N\right)$ is a nonnegative real sequence. By the supremum and the infimum properties, for every $N\geq 1$ there exist some integers $n_N$ and $m_N$ such that
\begin{eqnarray}
\sum_{k=n_N+1}^{n_N+N}\ \delta_k\ \geq\ N\left(\frac{1}{p}-\frac{1}{4} + \varepsilon_N\right) - 1,\  \mbox{or} \quad \sum_{k=m_N+1}^{m_N+N}\ \delta_k\ \leq\ N\left(\frac{1}{4}-\frac{1}{q} - \eta_N\right) + 1. \nonumber
\end{eqnarray}
Also by definition, for every $K\leq N$
\begin{eqnarray}
\sum_{k=n_N+1}^{n_N+K}\ \delta_k\ \leq\ K\left(\frac{1}{p}-\frac{1}{4} + \varepsilon_K\right),\  \mbox{or} \quad \sum_{k=m_N+1}^{m_N+K}\ \delta_k\ \geq\ K\left(\frac{1}{4}-\frac{1}{q} - \eta_K\right). \nonumber
\end{eqnarray}
It follows that 
\begin{eqnarray}
\sum_{k=n_N+K+1}^{n_N+N}\ \delta_k\ \geq\ \left(\frac{1}{p}-\frac{1}{4}\right)(N-K) + N\varepsilon_N-K\varepsilon_K-1, \nonumber
\end{eqnarray}
and also 
\begin{eqnarray}
\sum_{k=m_N+K+1}^{m_N+N}\ \delta_k\ \leq\ \left(\frac{1}{4}-\frac{1}{q}\right)(N-K) - N\eta_N+K\eta_K+1. \nonumber
\end{eqnarray}
Recall that, since the sequence $(\delta_n)$ is bounded we proved that
\begin{equation}
\left|A_{p,n,m}\right| \asymp \left\{
  \begin{array}{ll}
    \exp\left[-\left(\frac{1}{q}-\frac{1}{4}\right)|m-n|-\underset{k=n}{\overset{m-1}{\sum}}\delta_k\right], & \hbox{$m\geq n$;} \\[0.4cm]
     \exp\left[-\left(\frac{1}{p}-\frac{1}{4}\right)|m-n|+\underset{k=m}{\overset{n-1}{\sum}}\delta_k\right], & \hbox{$n\geq m$.}
  \end{array}
\right.
\end{equation}
Now, suppose that the sequence $\left(N\varepsilon_N\right)$ (or $\left(N\eta_N\right)$) contains a subsequence tending to infinity. This implies $$\left|A_{p,n_N+N+1,n_N+1}\right|\ \gtrsim\ \exp(N\varepsilon_N),\ \mbox{or}\quad \left|A_{p,n_N+1,n_N+N+1}\right|\ \gtrsim\ \exp(N\eta_N),$$
and the matrix $A_{p}=\left(A_{p,n,m}\right)$ cannot map $\ell^p$ continuously into itself.

Suppose now that the sequence $\left(N\varepsilon_N\right)$ is bounded, as in the previous proof we obtain that $|A_{p,n_N+N,n_N+K+1}| \asymp 1$ for every $1\leq K<N$. Analogously, if $\left(N\eta_N\right)$ is bounded we get again that  $|A_{p,n_N+1,n_N+N-K}| \asymp 1$ for every $1\leq K< N$. In both situations, the matrix $A_p=\left(A_{p,n,m}\right)_{n,m}$ contains an  increasing number of coefficients in one line which are bounded away from zero and hence $A_p$ cannot define a bounded operator in $\ell^p$. This completes the proof for the case $4/3<p<4$.\\

% The proof of the cases $1\leq p<4/3$ and $p>4$ is the same to that of $4/3< p< 4$, we just remark that
Remark that 
\begin{enumerate}[label=$\bullet$]
\item For $p>4$, we have 
\begin{equation}
\left|A_{p,n,m}\right| \asymp \left\{
  \begin{array}{ll}
    \exp\left[-\left(\frac{1}{4}-\frac{1}{p}\right)|m-n|-\underset{k=n}{\overset{m-1}{\sum}}\delta_k\right], & \hbox{$m\geq n$;} \\[0.4cm]
     \exp\left[-\left(\frac{1}{4}+\frac{1}{p}\right)|m-n|+\underset{k=m}{\overset{n-1}{\sum}}\delta_k\right], & \hbox{$n\geq m$.}
  \end{array}
\right.
\end{equation}
\item For $1\leq p<\frac{4}{3}$, we have 
\begin{equation}
\left|A_{p,n,m}\right| \asymp \left\{
  \begin{array}{ll}
    \exp\left[-\left(\frac{1}{q}+\frac{1}{4}\right)|m-n|-\underset{k=n}{\overset{m-1}{\sum}}\delta_k\right], & \hbox{$m\geq n$;} \\[0.4cm]
     \exp\left[-\left(\frac{1}{4}-\frac{1}{q}\right)|m-n|+\underset{k=m}{\overset{n-1}{\sum}}\delta_k\right], & \hbox{$n\geq m$.}
  \end{array}
\right.
\end{equation}
\end{enumerate}  
Using these estimates and following the same steps of the proof of the case $4/3<p<4$ one can check easily that the matrix $\left(A_{p,n,m}\right)_{n,m}$ cannot map boundedly $\ell^p$ into itself, whenever the condition on $\Delta_N$ (or on $\nabla_N$) is not satisfied, and hence $T_\Gamma^{-1}$ is not bounded too. This ends the proof of Theorem \ref{thm3}.
\end{enumerate}
% % % % % % % % % % % % % % % % % % % % % % % % % % % %
 % % % % % % % % % % % % % % % % % % % % % % % % % % % %
% % % % % % % % % % % % % % % % % % % % % % % % % % % % 
{
\section{Final Remarks}
In this section we give some final remarks about the results obtained in Theorems \ref{thm3} and \ref{thm2}. In fact, roughly speaking, in Theorem \ref{thm3} if $p$ increases to 4, then \ref{BB} implies that  $\Gamma=\{\gamma_n:=\lambda_ne^{\delta_n}e^{i\theta_n}\}_{n\geq 0}$ is a complete interpolating set for $\cF^4_\varphi$ if and only if \ref{aa} and \ref{bb} are satisfied and for some $N\geq 1$ we have $-\frac{1}{2}<\nabla_N(\Gamma)\leq\Delta_N(\Gamma)<0$.  $\nabla_N(\Gamma)$ and $\Delta_N(\Gamma)$ are the corresponding means to the sequence $\Gamma$. Now, if $\Gamma=e^{-1/4}\Sigma:=\{\sigma_n:=e^{-1/4}\lambda_ne^{s_n}e^{i\psi_n} \}_{n\geq0}$, this is equivalent to $\lambda_ne^{\delta_n}e^{i\theta_n}=\gamma_n=\sigma_n=e^{-1/4}\lambda_ne^{s_n}e^{i\psi_n}$, i.e. $\delta_n=s_n-1/4$, for every $n\geq 0$. Simple calculations show that 
\[\nabla_N(\Gamma)=\nabla_N(\Sigma)-\frac{1}{4},\ \mbox{and}\quad \Delta_N(\Gamma)=\Delta_N(\Sigma)-\frac{1}{4}.\]
Thus, \[-\frac{1}{2}<\nabla_N(\Gamma)\leq\Delta_N(\Gamma)<0\ \Longleftrightarrow\ -\frac{1}{4}<\nabla_N(\Sigma)\leq\Delta_N(\Sigma)<\frac{1}{4},\]
and hence the result recovered from \ref{BB} in Theorem \ref{thm3}, by tending $p$ to 4, is equivalent to that given in Theorem \ref{thm2}. 

On the other hand, if $p$ decreases to 4, then \ref{CC} ensures that  $\Gamma_{-1}=\{\gamma_n:=\lambda_ne^{\delta_n}e^{i\theta_n}\}_{n\geq -1}$ is a complete interpolating sequence for $\cF^4_\varphi$ if and only if \ref{aa} and \ref{bb} are verified and  $0<\nabla_N(\Gamma)\leq\Delta_N(\Gamma)<\frac{1}{2}$, for some $N\geq 1$. Now, let $\Gamma_{-1}=e^{-1/4}\Sigma:=\{e^{-1/4}\lambda_ne^{s_n}e^{i\psi_n}\}_{n\geq 0}$. %Following the same steps of the last calculations one cane show that % this is  equivalent to the fact that $\lambda_{n-1}e^{\delta_{n_k}-1}e^{i\theta_{n-1}}=e^{-1/4}\lambda_ne^{s_n}e^{i\psi_n}$, for every $n\geq 0$, i.e. $\delta_{n-1}=s_n+\frac{1}{4}$. Simple computations imply that $\nabla_N(\Gamma)=\nabla_N(\Sigma)+\frac{1}{4}$ and $\Delta_N(\Gamma)=\Delta_N(\Sigma)+\frac{1}{4}$. Therefore,
As above we obtain
\[0<\nabla_N(\Gamma)\leq\Delta_N(\Gamma)<\frac{1}{2}\ \Longleftrightarrow\ -\frac{1}{4}<\nabla_N(\Sigma)\leq\Delta_N(\Sigma)<\frac{1}{4}, \]
and hence the result obtained from \ref{CC} in Theorem \ref{thm3} is equivalent to that proved in Theorem \ref{thm2}. The case $p=4/3$ is also the same. This clarify the nonexistence of a certain discontinuity in the obtained results.
}% % % % % % % % % % % % % % % % % % % % % % % % % % % % % %

\subsection*{Acknowledgments}
%, they also gave helpful discussions, suggestions and remarks. The author is deeply grateful to them for their 
O. El-Fallah  and K. Kellay have read the article and suggested many improvements and simplifications. The author is deeply grateful to them and also to A. Baranov and A. Borichev for helpful remarks, as well as the referees for their suggestions and practical remarks.

%The author is deeply grateful to the Professors O. El-Fallah (Mohammed V university, faculty of sciences in Rabat) and K. Kellay (IMB, University of Bordeaux) for their helpful discussions, suggestions and remarks.

% % % % % % % % % % % % % % % % % % % % % % % % % % % % % % % % % % % % % % % % % % % % % %
% % % % % % % % % % % % % % % % % % % % % % % % % % % % % % % % % % % % % % % % % % % % % % % % % % % % % % % % % % % % % % % % % % % % % % % % % % % % 
%\bibliographystyle{plain}
%\bibliography{references}

\end{document}